\documentclass[10pt]{amsart}
\usepackage{graphicx}
\usepackage{amscd}
\usepackage{amsmath}
\usepackage{amsthm}
\usepackage{amsfonts}
\usepackage{amssymb}
\usepackage{mathrsfs}
\usepackage{bm}
\usepackage{enumerate}
\usepackage{amsrefs}
\usepackage{xcolor}
\usepackage[colorlinks, citecolor=blue, linkcolor=red, pdfstartview=FitB]{hyperref}
\usepackage[latin1]{inputenc}

\usepackage{marginnote}	
\usepackage{soul} 			

\newcommand{\bB}{{\mathbb{B}}}
\newcommand{\bC}{{\mathbb{C}}}

\newcommand{\bF}{{\mathbb{F}}}

\newcommand{\bN}{{\mathbb{N}}}

\newcommand{\bR}{{\mathbb{R}}}
\newcommand{\bS}{{\mathbb{S}}}

  \newcommand{\A}{{\mathcal{A}}}

  \newcommand{\D}{{\mathcal{D}}}
  \newcommand{\E}{{\mathcal{E}}}
  \newcommand{\F}{{\mathcal{F}}}
  
\renewcommand{\H}{{\mathcal{H}}}

\renewcommand{\L}{{\mathcal{L}}}
  \newcommand{\M}{{\mathcal{M}}}

  \newcommand{\X}{{\mathcal{X}}}

\newcommand{\fA}{{\mathfrak{A}}}

\newcommand{\fF}{{\mathfrak{F}}}

\newcommand{\fK}{{\mathfrak{K}}}

\newcommand{\fL}{{\mathfrak{L}}}

\newcommand{\fr}{{\mathfrak{r}}}

\newcommand{\fs}{{\mathfrak{s}}}
\newcommand{\fT}{{\mathfrak{T}}}

\newcommand{\fW}{{\mathfrak{W}}}

\newcommand{\fz}{{\mathfrak{z}}}

\newcommand{\rA}{\mathrm{A}}

\newcommand{\rC}{\mathrm{C}}

\newcommand{\eps}{\varepsilon}
\renewcommand{\phi}{\varphi}

\newcommand{\upchi}{{\raise.35ex\hbox{$\chi$}}}

\newcommand{\ol}{\overline}

\newcommand{\AB}{{\mathrm{A}(\mathbb{B}_d)}}

\newcommand{\HB}{{H^\infty(\mathbb{B}_d)}}

\newcommand{\qand}{\quad\text{and}\quad}

\newcommand{\spn}{\operatorname{span}}

\newcommand{\Band}{\operatorname{\mathscr{B}}}

\newcommand{\Hen}{\operatorname{Hen}}

\newcommand{\frk}[1]{\mathfrak{#1}}

\newcommand{\Ro}{\mathscr{R}_0}
\newcommand{\AC}{\operatorname{AC}}

\newcommand{\SG}{\operatorname{SG}}

\newtheorem{lemma}{Lemma}[section]
\newtheorem{theorem}[lemma]{Theorem}
\newtheorem{proposition}[lemma]{Proposition}
\newtheorem{corollary}[lemma]{Corollary}
\newtheorem{theoremx}{Theorem}

\theoremstyle{definition}
\newtheorem{definition}[lemma]{Definition}

\newtheorem{question}{Question}
\newtheorem{example}{Example}

\date{\today}
\author{Rapha\"el Clou\^atre}

\address{Department of Mathematics, University of Manitoba, Winnipeg, Manitoba, Canada R3T 2N2}

\email{raphael.clouatre@umanitoba.ca\vspace{-2ex}}
\author{Edward J. Timko}
\email{edward.timko@umanitoba.ca\vspace{-2ex}}
\thanks{R.C. was partially supported by an NSERC Discovery Grant. E.J.T. was partially supported by a PIMS postdoctoral fellowship.}

\title[Non-commutative Henkin theory]{Non-commutative measure theory: Henkin and analytic functionals on $\rC^*$-algebras}
\begin{document}
\begin{abstract}
Henkin functionals on non-commutative $\rC^*$-algebras have recently emerged as a pivotal link between operator theory and complex function theory in several variables. Our aim in this paper is characterize these functionals through a notion of absolute continuity, inspired by a seminal theorem of Cole and Range. To do this, we recast the problem as a question in non-commutative measure theory. We develop a Glicksberg--K\"onig--Seever decomposition of the dual space of a $\rC^*$-algebra into an absolutely continuous part and a singular part, relative to a fixed convex subset of states. Leveraging this tool, we show that Henkin functionals are absolutely continuous with respect to the so-called analytic functionals if and only if a certain compatibility condition is satisfied by the ambient weak-$*$ topology. In contrast with the classical setting, the issue of stability under absolute continuity is not automatic in this non-commutative framework, and we illustrate its key role in sharpening our description of Henkin functionals. Our machinery yields new insight when specialized to the multiplier algebras of the Drury--Arveson space and of the Dirichlet space, and to Popescu's noncommutative disc algebra. As another application, we make a contribution to the theory of non-commutative peak and interpolation sets.
\end{abstract}
\maketitle

\section{Introduction}

This paper is centred around certain distinguished elements of the dual of a $\rC^*$-algebra. Our motivation is the following classical situation.

Let $d\geq 1$ be a positive integer.  Let $\bB_d\subset \bC^d$ be the open unit ball and denote its topological boundary, the unit sphere, by $\bS_d$. The \emph{ball algebra} $\AB$ is the norm closure of the polynomials inside of the commutative $\rC^*$-algebra $\rC(\bS_d)$. We let $\sigma$ denote the unique rotation-invariant, regular, Borel probability measure on $\bS_d$ \cite[Remark, p.16]{rudin2008}. We view $\rC(\bS_d)$ as being embedded inside of the commutative von Neumann algebra $L^\infty(\bS_d,\sigma)$, and we analyze the triple
\begin{equation}\label{Eq:tripleintro}
\AB\subset \rC(\bS_d)\subset L^\infty(\bS_d,\sigma)
\end{equation}
from the point of view of dual spaces. 

First, we recall that the dual space of $\rC(\bS_d)$ can be identified isometrically with $M(\bS_d)$, the space of regular Borel measures on $\bS_d$. 
In this space are some distinguished elements, called \emph{Henkin measures},  that are particularly relevant for our purposes in this paper. Roughly speaking, these are the measures that are compatible with the inclusions \eqref{Eq:tripleintro}. More precisely, a measure $\mu\in M(\bS_d)$ is Henkin if 
\[
\lim_{n\to\infty}\int_{\bS_d}a_n d\mu=0
\]
whenever $(a_n)$ is a sequence in $\rA(\bB_d)$ that converges to $0$ in the weak-$*$ topology of $L^\infty(\bS_d,\sigma)$. 

Henkin measures were introduced in  \cite{henkin1968} as a tool to distinguish the ball algebra from the polydisc algebra (in possibly different dimensions), and are now more or less completely understood, thanks to seminal contributions of Valskii \cite{valskii1971}, Glicksberg and K\"onig--Seever \cite{glicksberg1967},\cite{konig1969}, and Cole--Range \cite{cole1972}. This elegant theory is neatly laid out in \cite[Chapter 9]{rudin2008}, where it culminates in the fact that a measure $\mu\in M(\bS_d)$ is Henkin if and only if it is absolutely continuous with respect to some regular Borel probability measure $\rho$ on $\bS_d$ that represents evaluation at the origin, in the sense that
\[
\int_{\bS_d}ad\rho=a(0), \quad a\in \rA(\bB_d).
\]

In the univariate case, the ball algebra holds the key to the structure of Hilbert space contractions, through von Neumann's inequality \cite{nagy2010}. It was noted, at least as early as \cite{eschmeier1997} and \cite{eschmeier2001}, that Henkin measures in higher dimensions also carry substantial operator theoretic information.  However, this picture is incomplete. A standard paradigm in multivariate operator theory dictates that in order to faithfully model the behaviour of several (commuting) operators, one is inexorably led to algebras of holomorphic functions that cannot be embedded in commutative $\rC^*$-algebras \cite{muller1993},\cite{arveson1998}. An analogue of Henkin measures adapted to this more general framework is thus required. 

Accordingly, the concept of Henkin functionals was explored in \cite{CD2016duality},\cite{CD2016abscont} for the multiplier algebra of the Drury--Arveson space, and subsequently refined and extended to a large family of operator algebras of holomorphic functions \cite{BHM2018},\cite{DH2020}. While important aspects of the aforementioned classical theory of Henkin measures for the ball algebra can be realized in this more general framework, some fundamental pieces remained elusive upon the appearance of \cite{CD2016duality}. Chief amongst these mysteries was the lack of a satisfactory analogue of the Cole--Range description of Henkin measures. Faced with this difficulty, a hopeful conjecture was formulated in \cite{CD2016duality}, proposing that the classical description of Henkin measures was still valid for the Drury--Arveson space. This was disproved soon thereafter in \cite{hartz2018henkin}, and to the best of our knowledge the problem of \emph{completely} characterizing Henkin functionals outside of the ball algebra is yet unresolved, and largely unexplored. The current paper initiates this program, motivated by the aforementioned operator theoretic considerations. Interestingly, Henkin measures in the setting of the Dirichlet space have recently been shown to be intimately connected to potential theory \cite{ChHartz2020}, providing further drive for our investigation.

We approach the problem by recasting it much more generally. Fix a von Neumann algebra $\fW$ and a unital norm-closed subalgebra $\A\subset \fW$. We define $\Hen_\fW(\A) \subset \A^*$ to be the subset of bounded linear functionals $\phi:\A\to \bC$ with the property that
$
\lim_i \phi(a_i)=0
$
whenever $(a_i)$ is a bounded net in $\A$ converging to $0$ in the weak-$*$ topology of $\fW$. 
Our main goal, given a unital $\rC^*$-algebra $\fT$ with $\A\subset \fT \subset \fW$, is then to elucidate the structure of the set 
\[
\{\phi\in \fT^*:\phi|_\A\in \Hen_\fW(\A)\}.
\]
Naturally, elements therein can be viewed as  non-commutative analogues of Henkin measures. For this purpose, we draw inspiration from the classical solution in \cite[Chapter 9]{rudin2008}, and seek to describe the set above by means of what one may call non-commutative measure theory. Although our approach seems essentially disjoint from it, we mention that such a theory was carved out in Pedersen's seminal work \cite{pedersen1966I},\cite{pedersen1966II},\cite{pedersen1966IIIIV}.

We now describe the organization of the paper. 

The preliminary Section \ref{S:AC} introduces a notion of absolute continuity and of singularity for functionals on $\rC^*$-algebras. It should be acknowledged that other non-commutative versions of the classical measure theoretic notions have been investigated previously in various contexts  \cite{sakai1965},\cite{PT1973},\cite{exel1990},\cite{vaes2001},\cite{GK2009}. The precise version that we require for our purposes appears to be different, however. Section \ref{S:AC} collect various technical facts on these topics that are required subsequently. 

Section \ref{S:GKS} is concerned with purely non-commutative measure theoretic questions, and sets the stage for the identification of Henkin functionals to come later. Therein, we prove a non-commutative analogue of the Glicksberg--K\"onig--Seever decomposition \cite[Theorem 9.4.4]{rudin2008}. This is done in two steps. First, we exhibit a Riesz decomposition of the dual of a $\rC^*$-algebra $\fT$ with respect to a norm-closed convex subset $\Delta$ of states on $\fT$. In more details, we show in Theorem \ref{T:ncRiesz} that
\[
\fT^*=\AC(\Delta)\oplus \SG(\Delta)
\]
where $\AC(\Delta)$ is the space of functionals on $\fT$ that are absolutely continuous with respect to some state in $\Delta$, while $\SG(\Delta)$ is the space of functionals on $\fT$ that are singular with respect to every element in $\Delta$. %
Second, using Akemann's non-commutative topology \cite{akemann1969}, we refine the decomposition when the convex subset $\Delta$ is in fact closed in the weak-$*$ topology (Theorem \ref{T:rainwater}). This is in line with Rainwater's contribution to the classical Glicksberg--K\"onig--Seever decomposition \cite{rainwater1969}. 

In Section \ref{S:ncHenkin}, we turn our attention to the main problem driving this paper, namely the identification of Henkin functionals. We recall the typical setup. Let $\fW$ be a von Neumann algebra, let $\fT\subset \fW$ be a unital $\rC^*$-algebra and let $\A\subset \fT$ be a unital norm-closed subalgebra.  Let $\Delta$ denote the norm closure in $\fT^*$ of the convex hull of $\{|\phi|:\phi\in \A^\perp,\|\phi\|=1\}$ (see Section \ref{S:AC} for a discussion on absolute values of functionals). We say that the triple $\A\subset \fT\subset \fW$ is \emph{analytic} if a bounded net $(b_i)$ in $\A$ converges to $0$ in the weak-$*$ topology of $\fW$ whenever 
\[
\lim_i \alpha(b_i)=0, \quad \alpha\in \AC(\Delta).
\]
We remark that, classically, measures on the unit circle that annihilate the disc algebra are said to be analytic, so for this reason we think of $\A^\perp$ as those analytic functionals on $\fT$; this explains our choice of terminology. The property of a triple being analytic is a salient feature of our investigation, and we exhibit concrete sufficient conditions for analyticity of the triple in Propositions \ref{P:anal} and \ref{P:coanal}. 

The following is one of our main results and can be found in Theorem \ref{T:nccr}, where we write
\[
\Band=\{\phi\in \fT^*:\phi|_\A\in \Hen_{\fW}(\A)\}.
\]

\begin{theoremx}\label{T:A}
The triple $\A\subset \fT\subset \fW$ is analytic if and only if $\Band\subset \AC(\Delta).$
\end{theoremx}

In fact, Theorem \ref{T:nccr} gives more information. Indeed, due to the potential lack of commutativity of the algebra $\fT$, certain asymmetries are pervasive throughout our work, and hence require us to prove a certain ``conjugate" version of the previous result. We refer the reader to Section \ref{S:ncHenkin} for details.

Now, it is natural to ask whether we should expect the equality $\Band=\AC(\Delta)$ to hold in Theorem \ref{T:A}. Interestingly, this depends on whether $\Band$ is what we call a \emph{band}, namely whether membership in $\Band$ is stable under absolute continuity. Once again, due to the lack of commutativity in our framework, we must consider both a left and a right version of this notion. The following is another one of our main results (see Theorem \ref{T:Henband}).

\begin{theoremx}\label{T:B}
Assume that the triple $\A\subset \fT\subset \fW$ is analytic. Then,  $\Band=\AC(\Delta)$  if and only if $\Band$ is a left band.
\end{theoremx}

Although we only state the left version of this theorem here, there is a corresponding right version which depends on the conjugate version of Theorem \ref{T:A} alluded to above.

In Section \ref{S:Ex}, we apply the general machinery developed thus far to three widely studied concrete examples of triples $\A\subset \fT\subset \fW$. We summarize our findings in the following.

\begin{theoremx}\label{T:C}
Let $d\geq 1$ be an integer. 
\begin{enumerate}[{\rm (i)}]
\item The classical triple
$
\AB\subset \rC(\bS_d)\subset L^\infty(\bS_d,\sigma)
$ is analytic and coanalytic. The corresponding set $\Band$ is a left and right band, and $\Band=\AC(\Delta)$. 

\item Let $\H$ be a regular, unitarily invariant, reproducing kernel Hilbert space on $\bB_d$. Let $\A(\H)\subset B(\H)$ denote the norm-closure of the polynomial multipliers and let $\fT(\H)=\rC^*(\A(\H))$. Then, the triple $\A(\H)\subset \fT(\H) \subset B(\H)$ is analytic and coanalytic. The corresponding set $\Band$ is a left and right band, and $\Band=\AC(\Delta)$. 

\item Let $\fA_d$ denote Popescu's non-commutative disc algebra acting on the full Fock space $\fF^2_d$ over $\bC^d$. Let $\fT_d=\rC^*(\fA_d)$. Then, the triple  $\fA_d\subset \fT_d \subset B(\fF^2_d)$ is analytic and coanalytic. The corresponding set $\Band$ is a right band but it is not a left band when $d>1$.
\end{enumerate}
\end{theoremx}

We emphasize that statement (ii) above covers a large family of important  spaces from complex function theory and harmonic analysis, including for instance the Drury--Arveson space in any finite dimension and the Dirichlet space on the disc. In particular, to the best of our knowledge, the previous result offers the first complete description of Henkin functionals for these spaces. This is particularly interesting in view of the work done in \cite{CD2016duality} and \cite{hartz2018henkin}, as described earlier.

At this point, the reader may have noted that Theorems \ref{T:A} and \ref{T:B} constitute somewhat of a departure from the classical Cole--Range theorem for the ball algebra. This discrepancy is addressed at length in Section \ref{S:Ex}.  At least for the ball algebra and for the non-commutative disc algebra, we establish variants of Theorems \ref{T:A} and \ref{T:B} that are closer  to the classical Cole--Range theorem. Roughly speaking, we verify that the set $\Band$ can also be realized as comprising the functionals on $\fT$ that are absolutely continuous with respect to a state  representing ``evaluation at the origin" on $\A$ (see Theorems \ref{T:crAB} and \ref{T:crAd}). However, at the time of this writing we do not know whether the corresponding statement is valid for $\A(\H)$, and are thus left with the following outstanding question.

\begin{question}\label{Q:crAH}
Let $\H$ be a regular, unitarily invariant, reproducing kernel Hilbert space on $\bB_d$. Let $\Ro$ denote the set of states $\rho$ on $\fT(\H)$ with the property that
\[
\rho(a)=a(0), \quad a\in \A(\H).
\]
Is it true that $\Band=\AC(\Ro)$?
\end{question}

Finally, in Section \ref{S:nullproj} we give an application of Theorem \ref{T:A} to the theory of non-commutative peak sets  \cite{hay2007},\cite{BHN2008},\cite{BR2011},\cite{blecher2013},\cite{BR2013}. We show in Theorem \ref{T:peakequiv} that given an analytic triple $\A\subset \fT\subset \fW$, a \emph{null} projection $q\in \fT^{**}$ which is closed in the sense of Akemann must necessarily be \emph{totally null}. Furthermore, we clarify the relationships between these conditions and the properties that $q$ be a non-commutative peak set, a non-commutative interpolation set, or a non-commutative vanishing locus. We also show that equivalences between these notions that hold in the classical case happen to fail more generally.


\section{Preliminaries on absolute continuity}\label{S:AC}
\subsection{Polar decomposition of functionals}\label{SS:polar}
Let $\fT$ be a $\rC^*$-algebra. Recall that the bidual $\fT^{**}$ can be given the structure of a von Neumann algebra that contains $\fT$ as a weak-$*$ dense $\rC^*$-subalgebra \cite[Theorem A.5.6]{BLM2004}.

%
Given a bounded linear functional $\phi:\fT\to \bC$, throughout the paper we let $\widehat{\phi}:\fT^{**}\to \bC$ denote its unique weak-$*$ continuous extension. We recall the details pertaining to the polar decomposition of $\phi$. By \cite[Theorem 1.14.4]{sakai1971}, there is a unique partial isometry $v_\phi\in \fT^{**}$ and a unique weak-$*$ continuous positive linear functional $\omega_\phi:\fT^{**}\to \bC$ with the property that
\[
\widehat{\phi}(\xi)=\omega_\phi(\xi v_\phi), \quad \xi\in \fT^{**}
\]
and
\[
\fT^{**}(I-v_\phi^*v_\phi)=\{\xi\in \fT^{**}: \omega_\phi(\xi^*\xi)=0\}.
\]
Henceforth, we use the following notations: $|\phi|=\omega_\phi$ and $\fs_\phi=v_\phi^*v_\phi$.
Note then that for $\xi\in \fT^{**}$ we have
\begin{align*}
|\phi|(\xi)&=|\phi|(\xi \fs_{\phi})=|\phi|(\xi v_\phi^*v_\phi)=\widehat{\phi}(\xi v_\phi^*).\\
\end{align*}
Summarizing, we have
\begin{equation}\label{Eq:absval}
\widehat\phi=|\phi|(\cdot v_\phi) \qand |\phi|=\widehat\phi(\cdot v^*_\phi).
\end{equation}
These formulas will occur implicitly in the sequel. The following standard uniqueness statement will be used frequently, so we record if for ease of reference.

\begin{lemma}\label{L:takunique}
Let $\fT$ be a $\rC^*$-algebra and let $\phi:\fT\to \bC$ be a bounded linear functional. Let $\omega:\fT^{**}\to \bC$ be a positive linear functional with $\|\omega\|=\|\phi\|$ and such that
\[
|\widehat\phi(\xi)|^2\leq \|\phi\| \omega(\xi\xi^*), \quad \xi\in \fT^{**}.
\]
Then, we have $\omega=|\phi|$.
\end{lemma}
\begin{proof}
This is \cite[Proposition III.4.6]{takesaki2002}.
\end{proof}

Next, we relate the notion of absolute value to the classical measure theoretic one; the following is well known, but we provide the details for the convenience of the reader.

\begin{lemma}\label{L:absvalmeasure}
Let $X$ be a compact Hausdorff space. Let $\mu$ be a regular Borel measures on $X$ and define a bounded linear functional $\phi_\mu:\rC(X)\to \bC$ by
\[
\phi_\mu(f)=\int_X fd\mu, \quad f\in \rC(X).
\]
Then, 
\[
|\phi_\mu|(f)=\int_X fd|\mu|, \quad f\in \rC(X).
\]
\end{lemma}
\begin{proof}
Define $\omega:\rC(X)\to \bC$ as 
\[
\omega(f)=\int_X fd|\mu|, \quad f\in \rC(X).
\]
An easy verification based on Kaplansky's density theorem reveals that $\widehat\omega$ is a positive linear functional on $\rC(X)^{**}$.
By the Radon--Nikodym theorem, there is a measurable function $\gamma:X\to\bC$ such that $|\gamma(x)|=1$ for every $x\in X$ and with the property that $d|\mu|=\gamma d\mu.$ Given $\theta\in \rC(X)^*$, there is a regular Borel measure $\tau_\theta$ such that 
\[
\theta(f)=\int_X fd\tau_\theta, \quad f\in \rC(X).
\]
Define
 $v\in \rC(X)^{**}$ as
\[
v(\theta)=\int_X \gamma d\tau_\theta, \quad \theta\in \rC(X)^*.
\]
It is readily verified that $v$ is unitary and that $\phi=\widehat\omega(\cdot v^*)$. In particular, this implies that $\|\omega\|=\|\phi\|$ and
\begin{align*}
|\widehat\phi(\xi)|^2=|\widehat\omega(\xi v^*)|^2\leq \|\phi\| \omega(\xi \xi^*)
\end{align*}
by the Schwarz inequality. Finally, Lemma \ref{L:takunique} implies $\omega=|\phi|$.
\end{proof}

The absolute value behaves well under changes of representations, as we record next.

\begin{lemma}\label{L:absvaluecomp}
Let $\fT$ be a $\rC^*$-algebra, let $\pi:\fT\to B(\H)$ be a $*$-representation and let $\phi:\pi(\fT)\to \bC$ be a bounded linear functional. Then, 
$
|\phi\circ \pi|=|\phi|\circ \pi^{**}.
$
\end{lemma}
\begin{proof}
Consider the weak-$*$ continuous $*$-homomorphism $\pi^{**}:\fT^{**}\to \pi(\fT)^{**}$. There exists a central projection $\fz\in \fT^{**}$ with the property that $\ker \pi^{**}=\fT^{**}(I-\fz)$. Then, the map 
\[
 \xi\fz\mapsto \pi^{**}(\xi), \quad \xi \in \fT^{**}
\]
implements a weak-$*$ continuous $*$-isomorphism between $ \fT^{**}\fz$ and $\pi(\fT)^{**}$.
Let $v\in \pi(\fT)^{**}$ be a partial isometry with the property that $\widehat\phi=|\phi|(\cdot v)$. Correspondingly, there is a partial isometry $u\in \fT^{**}\fz$ such that $\pi^{**}(u)=v$. 
Put $\omega=|\phi|\circ \pi^{**}$, which is a weak-$*$ continuous positive linear functional on $\fT^{**}$ with norm equal to $\|\phi\|$. Given $\xi\in \fT^{**}$, we find by the Schwartz inequality that
\begin{align*}
|(\widehat \phi\circ \pi^{**})(\xi)|^2&=| (|\phi| \circ \pi^{**})(\xi u)|^2\\
&=|\omega(\xi u)|^2\leq \omega(u^*u)\omega(\xi\xi^*)\\
&\leq \|\omega\| \omega(\xi\xi^*)=\|\phi\| \omega(\xi\xi^*)\\
&=\|\phi\circ \pi\|\omega(\xi\xi^*).
\end{align*}
Finally, Lemma \ref{L:takunique} implies that $|\phi\circ \pi|=\omega.$
%
%
%
%
%
\end{proof}

\subsection{Absolute continuity}\label{SS:AC}
We can now give our definition of absolute continuity. Let $\phi:\fT\to \bC$ and $\psi:\fT\to \bC$ be bounded linear functionals. We say that $\phi$ is \emph{absolutely continuous with respect to} $\psi$ and write $\phi\ll\psi$ if, given $\xi\in \fT^{**}$, we have that
\[
|\psi|(\xi^*\xi)=0 \quad \text{implies} \quad |\phi|(\xi^*\xi)=0.
\]
Equivalently, we see that $\phi\ll\psi$ if and only if $\fs_\phi\leq \fs_\psi$. 
We say that $\phi$ and $\psi$ are \emph{mutually singular} if $\fs_\phi\fs_\psi=0$.

The following fact is elementary but will be used repeatedly throughout, so we record it for ease of reference.

\begin{lemma}\label{L:statevanish}
Let $\fT$ be a $\rC^*$-algebra. Let  $\phi:\fT\to \bC$ and $\psi:\fT\to \bC$ be bounded linear functionals. Assume that $\phi$ is absolutely continuous with respect to $\psi$.  Let $\xi\in \fT^{**}$ such that $|\psi|(\xi^*\xi)=0$. Then, $\widehat\phi( \xi^* \eta)=0$ for every $\eta\in \fT^{**}$.
\end{lemma}
\begin{proof}
Choose a partial isometry $v\in \fT^{**}$ with the property that
$
\widehat\phi=|\phi|(\cdot v).
$
Let $\eta\in \fT^{**}$. By assumption, we have that $|\phi|(\xi^*\xi)=0$, so by the Schwarz inequality we find
\begin{align*}
|\widehat\phi(\xi^*\eta)|=| |\phi|(\xi^*\eta v)|\leq  \|\eta v\| \sqrt{|\phi|(\xi^*\xi)}=0.
\end{align*}
\end{proof}

We will  also require the following basic observation.

\begin{lemma}\label{L:ACmodule}
Let $\fT$ be a $\rC^*$-algebra and let $\psi:\fT\to \bC$ be a bounded linear functional. Let $\eta\in \fT^{**}$ and put $\phi=\widehat\psi(\cdot \eta)$.
Then, $\phi$ is absolutely continuous with respect to $\psi$.
\end{lemma}
\begin{proof}
Let $\xi\in \fT^{**}$ such that $|\psi|(\xi^*\xi)=0$. Choose partial isometries $v,w\in \fT^{**}$ with the property that
\[
\widehat{\psi}=|\psi|(\cdot v) \qand |\phi|=\widehat\phi(\cdot w).
\]
We find
\begin{align*}
|\phi|(\xi^*\xi)&=\widehat\phi(\xi^* \xi w)=\widehat\psi(\xi^*\xi w\eta)\\
&=|\psi|(\xi^*\xi w\eta v).
\end{align*}
Applying the Schwarz inequality for the positive linear functional $|\psi|$, we infer
\begin{align*}
\left| |\phi|(\xi^*\xi)\right|^2&= \left||\psi|(\xi^*\xi w\eta v) \right|^2\\
&\leq |\psi|(\xi^*\xi) |\psi|( v^* \eta^* w^* \xi^* \xi w \eta v)\\
&=0.
\end{align*}
We conclude that $\phi$ is absolutely continuous with respect to $\psi$, as desired.
\end{proof}

Next, we illustrate with a simple example that the previous pair of results are truly one-sided.

\begin{example}\label{E:ACmodule}
Let $\fT=B(\bC^2)$ and let $\{e_1,e_2\}$ be the standard orthonormal basis of $\bC^2$. Because $\fT$ is finite-dimensional, we see that $\fT=\fT^{**}$. Let $\psi:\fT\to \bC$ be the positive linear functional defined as
\[
\psi(t)=\langle te_1,e_1\rangle, \quad t\in \fT.
\]
Let $u\in \fT$ be the unitary operator 
\[
u=\begin{bmatrix}
0 & 1\\
1 & 0
\end{bmatrix}.
\]
Define bounded linear functionals  $\phi$ and $\theta$ on $\fT$ as $\theta=\psi(u \cdot)$ and $\phi=\psi(\cdot u)$, so that
\[
\theta(t)=\langle te_1,e_2 \rangle \qand \phi(t)=\langle te_2,e_1 \rangle
\]
for every $t\in \fT$.
We also consider the positive linear functional $\omega$ on $\fT$ defined as $\omega=\psi(u^* \cdot u)$, so that
\[
\omega(t)=\langle te_2,e_2\rangle, \quad t\in \fT.
\]
It is readily checked that $|\theta|=\omega$. Let $\xi\in \fT$ be the orthogonal projection onto $\bC e_2$. Then, 
\[
|\psi|(\xi^*\xi)=0
\]
yet
\[
|\theta|(\xi^*\xi)=1.
\]
We conclude that $\theta$ and $\omega$ are not absolutely continuous with respect to $\psi$ (compare with the statement of Lemma \ref{L:ACmodule}). On the other hand, $\phi$ is absolutely continuous with respect to $\psi$ by virtue of Lemma \ref{L:ACmodule}, yet
\[
\phi(u \xi)=1
\]
(compare with the statement of Lemma \ref{L:statevanish}).
\qed
\end{example}
We remark here that our notion of absolute continuity differs from that considered in \cite{exel1990}. Indeed, according to the definition adopted therein, all states are mutually absolutely continuous in the previous example since all non-zero $*$-representations of the $\rC^*$-algebra $B(\bC^2)$ are unitarily equivalent to the identity representation, and hence injective.

Example \ref{E:ACmodule} illustrates that care must be taken when working with absolute continuity in the noncommutative context. Things are simpler in the commutative case, where our definition of absolute continuity coincides with the classical measure theoretic one. 

\begin{lemma}\label{L:ACmeasure}
Let $X$ be a compact Hausdorff space. Let $\mu$ and $\nu$ be regular Borel measures on $X$ and define bounded linear functionals  $\phi$ and $\psi$ on  $\rC(X)$ as
\[
\phi(f)=\int_X fd\mu \qand \psi(f)=\int_X fd\nu
\]
for every $f\in \rC(X)$. Then, $\phi\ll\psi$ as functionals if and only if $\mu\ll\nu$ as measures.
\end{lemma}
\begin{proof}
Note first that 
\[
|\phi|(f)=\int_X fd|\mu| \qand |\psi|(f)=\int_X fd|\nu|
\]
for every $f\in \rC(X)$ by virtue of Lemma \ref{L:absvalmeasure}.

Assume next that $\phi\ll\psi$. Let $E\subset X$ be a measurable set such that $|\nu|(E)=0$. Given $\theta\in \rC(X)^*$, there is a regular Borel measure $\tau_\theta$ such that 
\[
\theta(f)=\int_X fd\tau_\theta, \quad f\in \rC(X).
\]
We may thus define $\xi_E\in \rC(X)^{**}$ as
\[
\xi_E(\theta)=\tau_\theta(E), \quad \theta\in \rC(X)^*. 
\]
It is readily verified that $\xi_E$ is a positive element. Observe that $|\psi|(\xi_E)=|\nu|(E)=0$. Since $\phi\ll\psi$, it follows that $|\mu|(E)=|\phi|(\xi_E)=0$. We conclude that $\mu\ll\nu$. 

Conversely, assume that $\mu\ll\nu$. Let $\xi\in \rC(X)^{**}$ such that $|\psi|(\xi^*\xi)=0$. By the Radon--Nikodym theorem, there is $r\in L^1(X,|\nu|)$ with the property that
\[
|\phi|(f)=\int_X frd|\nu|, \quad f\in \rC(X).
\]
Let $\eps>0$ and choose $s\in \rC(X)$ with the property that $\|r-s\|_{L^1(X,|\nu|)}<\eps$. We infer that $\||\phi|-|\psi|(\cdot s)\|_{\rC(X)^*}<\eps$. On the other hand, $|\psi|(\cdot s)\ll \psi$ by Lemma \ref{L:ACmodule}, whence $|\psi|(\xi^*\xi s)=0$ by Lemma \ref{L:statevanish} and
\[
|\phi|(\xi^*\xi)\leq \eps\|\xi\|^2+||\psi|(\xi^*\xi s)|=\eps\|\xi\|^2.
\]
We conclude that $|\phi|(\xi^*\xi)=0$, so indeed $\phi\ll\psi$.
\end{proof}

The existence of an analogue of the Radon--Nikodym derivative can be a subtle issue in general von Neumann algebras; see for instance \cite{sakai1965},\cite{PT1973},\cite{vaes2001}. Nevertheless, we can prove an elementary approximate version that is sufficient for our purposes in this paper. Given a bounded linear functional $\phi:\fT\to \bC$, we define the adjoint functional $\phi^\dagger:\fT\to \bC$ as
\[
\phi^\dagger(t)=\ol{\phi(t^*)}, \quad t\in \fT.
\]

\begin{theorem}\label{T:RN}
Let $\fT$ be a $\rC^*$-algebra and let $\phi$ and $\psi$ be bounded linear functionals on $\fT$. Then, the following statements hold.
\begin{enumerate}[{\rm (i)}]
\item The functional $\phi$ is absolutely continuous with respect to $\psi$ if and only $\phi$ belongs to the norm closure of 
$
\{\psi(\cdot t):t\in \fT\}.
$

\item The functional $\phi^\dagger$ is absolutely continuous with respect to $\psi^\dagger$ if and only $\phi$ belongs to the norm closure of 
$
\{\psi(t\cdot ):t\in \fT\}.
$
\end{enumerate}

\end{theorem}
\begin{proof}
(i) Put 
\[
\D_\psi=\{\psi(\cdot t):t\in \fT\}\subset \fT^*. 
\] 
Assume first that $\phi$ is absolutely continuous with respect to $\psi$. Let $\xi\in \fT^{**}$ be such that $\widehat\delta(\xi)=0$ for every $\delta\in \D_\psi$. This implies that
\[
|\psi|(\xi \eta)=0, \quad \eta\in \fT^{**}
\]
and in particular $|\psi|(\xi \xi^*)=0$. By assumption, we find $\widehat\phi(\xi)=0$ upon invoking Lemma \ref{L:statevanish}. The Hahn-Banach theorem then allows us to conclude that $\phi$ lies in the norm closure of $\D_\psi$. 

Conversely, assume that $\phi$ lies in the norm closure of $\D_\psi$.  Let $\xi\in \fT^{**}$ such that $|\psi|(\xi^*\xi)=0$. Let $\eps>0$. There is $t\in \fT$ such that
\[
\|\phi-\psi(\cdot t)\|<\eps.
\]
 Let $v\in \fT^{**}$ be a partial isometry such that $|\phi|=\widehat \phi(\cdot v)$. It follows from Lemma \ref{L:statevanish} that $\widehat\psi(\xi^*\xi vt)=0$, so that
\begin{align*}
|\phi|(\xi^*\xi)&=\widehat\phi(\xi^*\xi v)\leq |\widehat\psi(\xi^*\xi vt)|+\eps \|\xi\|^2=\eps\|\xi\|^2.
\end{align*}
Because $\eps$ is arbitrary, we conclude that $|\phi|(\xi^*\xi)=0$, whence $\phi$ is absolutely continuous with respect to $\psi$.

(ii) This follows immediately from (i) upon taking adjoints.
\end{proof}

When $\frak{T}$ is commutative, it readily follows from Theorem \ref{T:RN} that $\phi\ll\psi$ and $\phi^\dagger\ll \psi^\dagger$ are equivalent. 
However, we caution the reader that absolute continuity is generally not preserved upon taking adjoints. Indeed, Example \ref{E:ACmodule} exhibits a state $\psi$ and a bounded linear functional $\phi$ on $B(\bC^2)$ such that $\phi\ll\psi$  yet $\phi^\dagger$ is \emph{not} absolutely continuous with respect to $\psi=\psi^\dagger$.


\section{Non-commutative measure theory: decomposition of the dual space}\label{S:GKS}

In this section, we delve into non-commutative measure theory, in the sense that we study dual spaces of arbitrary $\rC^*$-algebras. Our goal is to establish a direct sum decomposition of these dual spaces, where the splitting is performed using the notions of absolute continuity and of singularity introduced in Section \ref{S:AC}.

Let $\fT$ be a unital $\rC^*$-algebra. Recall that the state space of $\fT$ is the set of unital positive linear functionals on $\fT$. Let $\Delta$ be a norm-closed convex subset of the state space of $\fT$. We let $\AC(\Delta)$ denote the subset of those bounded linear functionals $\phi:\fT\to\bC$ that are absolutely continuous with respect to some $\delta\in \Delta$. We start by establishing some elementary properties of $\AC(\Delta)$.

\begin{lemma}\label{L:ACclosed}
The set $\AC(\Delta)$ is a norm-closed subspace of $\fT^{*}$.
\end{lemma}
\begin{proof}
Let $\alpha,\beta\in \AC(\Delta)$ and let $c,d\in \bC$. Then, there are states $\delta,\eps\in \Delta$ with the property that $\alpha\ll\delta$ and $\beta\ll\eps$. Choose partial isometries $u,v,w\in \fT^{**}$ with the property that
\[
|c\alpha+d\beta|=(c\widehat\alpha+d\widehat \beta)(\cdot u), \quad \widehat\alpha=|\alpha|(\cdot v) \qand \widehat\beta=|\beta|(\cdot w).
\]
Since $\Delta$ is assumed to be convex, we see that the state $\phi=(\delta+\eps)/2$ lies in $\Delta$. Let $\xi\in \fT^{**}$ with the property that $\widehat{\phi}(\xi^*\xi)=0$. Then, 
\[
\widehat\delta(\xi^*\xi)=\widehat\eps(\xi^*\xi)=0.
\]
Applying Lemma \ref{L:statevanish}, we see that 
\[
\widehat\alpha(\xi^*\xi uv)=\widehat\beta(\xi^*\xi uw)=0.
\]
We infer
\begin{align*}
|c\alpha+d\beta|(\xi^*\xi)&=(c\widehat\alpha+d\widehat\beta)(\xi^*\xi u)\\
&=c|\alpha|(\xi^*\xi u v)+d|\beta|(\xi^*\xi u w)\\
&=0
\end{align*}
so that $c\alpha+d\beta\ll\phi$. We conclude that $c\alpha+d\beta\in \AC(\Delta)$, whence $\AC(\Delta)$ is a subspace of $\fT^*$. 

It only remains to show that $\AC(\Delta)$ is closed in the norm topology of $\fT^*$. To see this, let $(\alpha_n)$ be a sequence in $\AC(\Delta)$ converging in norm to some $\alpha\in \fT^*$. For each $n$, there is $\delta_n\in \Delta$ with the property that $\alpha_n\ll\delta_n$. Put $\delta=\sum_{n=1}^\infty 2^{-n}\delta_n$. Since $\Delta$ is convex and norm-closed, we see that $\delta\in \Delta$. Let $\xi\in \fT^{**}$ such that $\widehat\delta(\xi^*\xi)=0$. In particular, we see that $\widehat{\delta_n}(\xi^*\xi)=0$ and thus $|\alpha_n|(\xi^*\xi)=0$ for every $n$. Choose a partial isometry $v\in \fT^{**}$ such that $|\alpha|=\widehat\alpha(\cdot v)$.
By Lemma \ref{L:statevanish}, we see that
\[
\widehat{\alpha_n}(\xi^*\xi v)=0
\]
for every $n$. We conclude that
\begin{align*}
|\alpha|(\xi^*\xi)&=\widehat\alpha(\xi^*\xi v)=\lim_{n\to\infty}\widehat{\alpha_n}(\xi^*\xi v)=0.
\end{align*}
Consequently, $\alpha\ll\delta$ so indeed $\alpha\in \AC(\Delta)$.
\end{proof}

Next, we show that $\AC(\Delta)$ gives rise to a projection in $\fT^{**}$.

\begin{lemma}\label{L:ACideal}
There is a self-adjoint projection $\fr_\Delta\in \fT^{**}$ with the property that 
\[
\AC(\Delta)^\perp=(I-\fr_\Delta)\fT^{**} \qand \AC(\Delta)=((I-\fr_\Delta)\fT^{**})_\perp.
\]
\end{lemma}
\begin{proof}
First note that $\AC(\Delta)^\perp$ is clearly a weak-$*$ closed subspace of $\fT^{**}$. We claim that in fact it is a right ideal.  Let $\xi\in \AC(\Delta)^\perp$ and let $\eta\in \fT^{**}$. It follows from Lemma \ref{L:ACmodule} that the functional
\[
t\mapsto \widehat\phi(\cdot \eta), \quad t\in \fT
\]
lies in $\AC(\Delta)$ whenever $\phi\in \AC(\Delta)$. Thus, we find
\[
\widehat\phi(\xi\eta)=0, \quad \phi\in \AC(\Delta)
\]
which is equivalent to $\xi\eta\in \AC(\Delta)^\perp$. We conclude that indeed $\AC(\Delta)^\perp$ is a weak-$*$ closed right ideal of $\fT^{**}$, so by virtue of  \cite[Proposition 1.10.1]{sakai1971}, there is a self-adjoint projection $\fr_\Delta\in \fT^{**}$ with the property that 
\[
\AC(\Delta)^\perp=(I-\fr_\Delta)\fT^{**}.
\]
Finally, since $\AC(\Delta)$ is a norm-closed subspace by Lemma \ref{L:ACclosed}, the Hahn--Banach theorem implies that
\[
\AC(\Delta)=((I-\fr_\Delta)\fT^{**})_\perp.
\]
\end{proof}

Let $\SG(\Delta)$ denote those  bounded linear functionals $\phi:\fT\to\bC$ with the property that $\fs_\phi\fs_\delta=0$ for every $\delta\in \Delta$. In other words, $\SG(\Delta)$ consists of those functionals that are mutually singular with every element in $\Delta$. It follows trivially then  that $\AC(\Delta)\cap \SG(\Delta)=\{0\}$. Our next aim is to show that 
\[
\fT^*=\AC(\Delta)+\SG(\Delta).
\]
We need a few preliminary observations. 

\begin{lemma}\label{L:rproj}
We have that $\fr_\Delta=\vee_{\delta\in \Delta}\fs_{\delta}.$
\end{lemma}
\begin{proof}
Let $\delta\in \Delta$. Since $\delta$ is a state in $\AC(\Delta)$, it follows from Lemma \ref{L:ACideal} that $\widehat\delta(I-\fr_\Delta)=0$, whence $\fs_\delta\leq \fr_\Delta$. We conclude that $\vee_{\delta\in \Delta}\fs_{\delta}\leq \fr_\Delta$. Conversely, let $\xi\in \fT^{**}\left( I-\vee_{\delta\in \Delta}\fs_\delta\right).$ Then, we find $\widehat\delta(\xi^*\xi)=0$ for every $\delta\in \Delta$, whence $|\alpha|(\xi^*\xi)=0$  for every $\alpha\in \AC(\Delta)$. Thus,  $\widehat\alpha(\xi^*)=0$ for each $\alpha\in \AC(\Delta)$ by Lemma \ref{L:statevanish}, which in turn means that $\xi^*\in \AC(\Delta)^\perp$. By Lemma \ref{L:ACideal}, we conclude that $\xi^*=(I-\fr_\Delta)\xi^*$, and so $\xi=\xi(I-\fr_\Delta)$. This shows that
\[
\fT^{**}\left( I-\vee_{\delta\in \Delta}\fs_\delta\right)\subset \fT^{**}( I-\fr_\Delta)
\]
and thus $\fr_\Delta\leq \vee_{\delta\in \Delta}\fs_\delta$.
\end{proof}

We can now show how to detect membership in $\AC(\Delta)$ or $\SG(\Delta)$.

\begin{lemma}\label{L:split}
Let $\phi:\fT\to \bC$ be a bounded linear functional. Then, the following statements hold.
\begin{enumerate}[{\rm (i)}]
%
\item We have that $\phi\in \AC(\Delta)$ if and only if 
\[
\phi(t)=\widehat\phi(\fr_\Delta t), \quad t\in \fT.
\]

 \item We have that $\phi\in \SG(\Delta)$ if and only if 
\[
\phi(t)=\widehat\phi((I-\fr_\Delta) t), \quad t\in \fT.
\]

\item We have that $\SG(\Delta)=(\fr_\Delta \fT^{**})_\perp$. 
\end{enumerate}

\end{lemma}
\begin{proof}
(i) Let $\theta:\fT\to \bC$ be defined as
\[
\theta(t)=\widehat\phi(\fr_\Delta t), \quad t\in \fT.
\]
If $\phi\in \AC(\Delta)$, then $\phi\in ((I-\fr_\Delta)\fT^{**})_\perp$ by Lemma \ref{L:ACideal}, whence
\[
\widehat{\phi}((I-\fr_\Delta)\xi)=0, \quad \xi\in \fT^{**}.
\]
This clearly implies that $\theta=\phi$. Conversely, if $\theta=\phi$, then
\[
\widehat{\phi}((I-\fr_\Delta)\xi)=\widehat\theta((I-\fr_\Delta)\xi)=0, \quad \xi\in \fT^{**}.
\]
We conclude that $\phi\in ((I-\fr_\Delta)\fT^{**})_\perp$, whence $\phi\in \AC(\Delta)$ by Lemma \ref{L:ACideal}.

%
%

(ii) Let $\theta\in \fT^*$ be defined as
\[
\theta(t)=\widehat\phi((I-\fr_\Delta) t), \quad t\in \fT.
\]
If $\phi\in \SG(\Delta)$, then $\fs_\phi\fs_\delta=0$ for every $\delta\in \Delta$, and thus $\fs_\phi\fr_\Delta=0$ by Lemma \ref{L:rproj}. We conclude that $|\phi|(\fr_\Delta)=0$, whence
\[
\widehat{\phi}(\fr_\Delta t)=0, \quad t\in \fT
\]
by Lemma \ref{L:statevanish}. We infer that $\phi=\theta$. Conversely, assume $\theta=\phi$ and choose a partial isometry $v\in \fT^{**}$ such that $|\phi|=\widehat\phi(\cdot v)$. We find
\[
|\phi|(\xi)=\widehat\phi(\xi v)=\widehat\theta(\xi v)=\widehat\phi((I-\fr_\Delta)\xi v), \quad \xi\in \fT^{**}.
\]
This shows that $|\phi|(\fr_\Delta)=0$, whence $\fs_\phi \fr_\Delta=0$. Invoking Lemma \ref{L:rproj}, we see that $\fs_\phi\fs_\delta=0$ for every $\delta\in \Delta$.  Therefore $\phi\in \SG(\Delta)$.

(iii) This is an immediate consequence of (ii).
\end{proof}

We now arrive at the main result of this section, which is a non-commutative version of what is sometimes referred to as the Riesz decomposition theorem (see \cite{konig1969}.)

\begin{theorem}\label{T:ncRiesz}
Let $\fT$ be a unital $\rC^*$-algebra and let $\Delta$ be a norm-closed convex subset of the state space of $\fT$. Put $\fr_\Delta=\vee_{\delta\in \Delta}\fs_\delta$. Let
\[
\AC(\Delta)=\{\phi\in \fT^*:\phi\ll\delta \text{ for some } \delta\in \Delta\}
\]
and
\[
\SG(\Delta)=\{\phi\in \fT^*:\fs_\phi\fs_\delta=0 \text { for every } \delta\in \Delta\}.
\]
Then, we have
\[
\AC(\Delta)=((I-\fr_\Delta)\fT^{**})_\perp
\]
\[
\SG(\Delta)=(\fr_\Delta \fT^{**})_\perp
\]
and
\[
\fT^*=\AC(\Delta)+\SG(\Delta).
\]
\end{theorem}
\begin{proof}
This follows readily from Lemmas \ref{L:ACideal} and \ref{L:split}, along with the decomposition
\[
\phi=\widehat \phi(\fr_\Delta \cdot)+\widehat\phi((I-\fr_\Delta)\cdot)
\]
for every $\phi\in \fT^{*}$.
\end{proof}
 

\subsection{The Rainwater refinement}
When the set $\Delta$ in the previous development is in fact closed in the weak-$*$ topology of $\fT^*$, more can be said about the decomposition from Theorem \ref{T:ncRiesz}, and we get a statement closer to the classical Glicksberg--K\"onig--Seever theorem \cite{glicksberg1967},\cite{konig1969} (see also \cite[Theorem 9.4.4]{rudin2008}). The refinement we seek identifies a projection with good topological properties on which
the singular part of the Riesz decomposition of a functional must be concentrated. To make this more precise, we must recall Akemann's non-commutative topology for projections in the bidual of a $\rC^*$-algebra \cite{akemann1969},\cite{akemann1970left}.

Let $\fT$ be a $\rC^*$-algebra and let $p\in \fT^{**}$ be a projection. We say that $p$ is \emph{open} if there is a net  of positive contractions in $\fT$ that increases to $p$ in the weak-$*$ topology of $\fT^{**}$. Further, we say that a projection $q\in \fT^{**}$ is \emph{closed} if $I-q$ is open in the previous sense. If $t\in \fT$ is a self-adjoint element and $\chi$ is the characteristic function of some interval $J\subset \bR$, then it is easily verified that $\chi(t)$ is a projection in $\fT^{**}$ which is open (respectively, closed), whenever $J$ is open (respectively, closed) (see \cite[Proposition II.3]{akemann1969} for an argument). In the special case where $\fT=\rC(X)$ for some compact Hausdorff space $X$,  it can be verified that a projection $p\in \fT^{**}$ is open (respectively, closed) if and only if there exists an open (respectively, closed) subset $F\subset X$ with $p=\chi_F$.

In preparation for proving the desired refinement of our non-commutative Riesz decomposition, we require the following elementary fact.

\begin{lemma}\label{L:approx}
Let $\fT$ be a unital $\rC^*$-algebra and let $\Delta$ be a set of states on $\fT$. Let $\phi:\fT\to \bC$ be a linear functional with $\|\phi\|=1$ and let $0<\eps<1$. Assume that there is a positive contraction $t\in \fT$  such that $|\phi|(t)\geq 1-\eps$ and $\delta(t)<\eps$ for every $\delta\in \Delta$. Then, there is a closed projection $p\in \fT^{**}$ with the property that $|\phi|(p)\geq 1-\sqrt{\eps}$ and $\widehat\delta(p)<\eps/(1-\sqrt{\eps})$ for every $\delta\in \Delta$. 
\end{lemma}
\begin{proof}
Let $\chi:[0,1]\to [0,1]$ be the characteristic function of the interval $[1-\sqrt{\eps},1]$ and let $p=\chi(t)$. Then, $p$ is a closed projection in $\fT^{**}$ satisfying
\[
(1-\sqrt{\eps})p\leq t\leq (1-\sqrt{\eps})I+\sqrt{\eps}p.
\]
We obtain
\[
\widehat\delta(p)\leq \frac{\delta(t)}{1-\sqrt{\eps}}<\frac{\eps}{1-\sqrt{\eps}}
\]
for every $\delta\in \Delta$. Moreover, using that $|\phi|(I)=\|\phi\|=1$ we find
\begin{align*}
|\phi|(p)&\geq \frac{1}{\sqrt{\eps}}|\phi|(t-(1-\sqrt{\eps})I)\geq \frac{\sqrt{\eps}-\eps}{\sqrt{\eps}}=1-\sqrt{\eps}.
\end{align*}
\end{proof}

We now obtain our desired result, by adapting an idea of Rainwater \cite{rainwater1969} (see also \cite[Lemma 9.4.3]{rudin2008}). Roughly speaking, it says that a singular functional must be concentrated on a projection of type $F_\sigma$ that is negligible for $\Delta$.

\begin{theorem}\label{T:rainwater}
Let $\fT$ be a unital $\rC^*$-algebra and let $\Delta$ be a weak-$*$ closed convex subset of the state space of $\fT$. Let $\phi\in \SG(\Delta)$. Then, there is a countable collection $\F\subset \fT^{**}$ of closed projections  such that if we put $p=\vee\{q:q\in \F\}$, then $|\phi|(p)=|\phi|(I) $ and $\widehat\delta(p)=0$ for every $\delta\in \Delta$.
\end{theorem}
\begin{proof}
We may assume that $\|\phi\|=|\phi|(I)=1$. Consider the convex subset $G=\{t\in \fT:0\leq t\leq I\}$. Define a map $\Phi:G\times \Delta\to \bC$ as
\[
\Phi(t,\delta)=|\phi|(I-t)+\delta(t), \quad t\in G, \delta\in \Delta.
\]
Let $\eps>0$ and fix $\delta \in \Delta$.  By the Kaplansky density theorem, there is $t\in G$ with the property that
\[
| |\phi|(t)-|\phi|(\fs_\phi)|<\eps \qand |\delta(t)-\widehat\delta(\fs_\phi)|<\eps.
\]
Observe that $\fs_\phi\fs_\delta=0$ by assumption, so we obtain
\[
||\phi|(t)-1|<\eps \qand |\delta(t)|<\eps.
\]
We infer that $\Phi(t,\delta)<2\eps.$ Since $\eps>0$ and $\delta\in \Delta$ were arbitrary, this shows that
\[
\sup_{\delta\in \Delta}\inf_{t\in G}\Phi(t,\delta)=0.
\]
On the other hand, it follows from the minimax theorem \cite[Theorem 9.4.2]{rudin2008} that
\[
\sup_{\delta\in \Delta}\inf_{t\in G}\Phi(t,\delta)=\inf_{t\in G}\sup_{\delta\in \Delta}\Phi(t,\delta)
\]
so in fact we find
\[
\inf_{t\in G}\sup_{\delta\in \Delta}\Phi(t,\delta)=0.
\]
Given a positive integer $n$ we may thus find $t_{n}\in G$ such that
\[
\Phi(t_{n},\delta)<\frac{1}{4^{n}}, \quad \delta\in \Delta
\]
which is equivalent to
\[
|\phi|(I-t_{n})+\delta(t_{n})<\frac{1}{4^{n}}, \quad \delta\in \Delta.
\]
In particular, since $0\leq t_n\leq I$ we see that 
\[
|\phi|(t_{n})>1-\frac{1}{4^{n}} 
\]
and
\[
\delta(t_{n})<\frac{1}{4^{n}}, \quad \delta\in \Delta. 
\]
Apply now Lemma \ref{L:approx} to find a closed projection $r_{n}\in \fT^{**}$ with the property that
\[
|\phi|(r_{n})>1-\frac{1}{2^{n}} \qand \widehat\delta(r_{n})<\frac{1}{4^{n}}\frac{1}{1-1/2^n}\leq \frac{1}{2^n},
\]
whence 
\[ |\phi|(I-r_n)<\frac{1}{2^n} \qand \widehat\delta(I-r_n)\geq 1-\frac{1}{2^n} \]
 for every $\delta\in \Delta$. For each $m\in \bN$, we now put $q_m=\vee_{n=m+1}^\infty (I-r_n)$, which is open. We also let $p=\vee_{m=1}^\infty(I- q_m)$. 
For each $m\in \bN$, we find
\[
|\phi|(I-p)\leq |\phi|(q_m)\leq \sum_{n=m+1}^\infty |\phi|(I-r_{n})<\sum_{n=m+1}^\infty \frac{1}{2^n}=\frac{1}{2^m}
\]
whence $|\phi|(p)=1$. Finally,  let $\delta\in \Delta$. For each $m\in \bN$  and each $n\geq m+1$ we have
\[
\widehat\delta(q_m)\geq \widehat\delta(I-r_n)\geq 1-\frac{1}{2^n}
\]
so that  $\widehat\delta(q_m)=1$ and $\fs_\delta\leq q_m$. We infer that $I-\fs_\delta \geq p$, so that $\widehat\delta(p)=0$.
\end{proof}


\section{Non-commutative Henkin theory}\label{S:ncHenkin}

\subsection{Henkin functionals}\label{SS:Henkinfunct}
For much of the remainder of the paper, we will be concerned with a special class of functionals, which we now define. 

Let $\fW$ be a von Neumann algebra and let $\X\subset \fW$ be a norm-closed subspace.
A  bounded linear functional $\phi:\X\to \bC$ is said to be \emph{Henkin relative to $\fW$} if  $\lim_i \phi(a_i)=0$ for every bounded net $(a_i)$ in $\X$ converging to $0$ in the weak-$*$ topology of $\fW$.  We let  $\Hen_{\fW}(\X)$ denote the subspace of $\X^*$ consisting of those functionals that are Henkin relative to $\fW$. It is readily verified that this subspace is norm closed.

Next, let $\E_\fW(\X)$ denote the set of those bounded linear functionals on $\X$ which extend weak-$*$ continuously to $\fW$. It is clear that $\E_\fW(\X)\subset \Hen_\fW(\X)$. Equality holds in some instances as we show in the next result, which is directly adapted from \cite[Lemma 1.1]{eschmeier1997}.

\begin{proposition}\label{P:Esc}
Let $\fW$ be a von Neumann algebra and let $\X\subset \fW$ be a norm-closed subspace. Let $\fL$ denote the closure of $\X$ in the weak-$*$ topology of $\fW$. 
Assume that the closed unit ball of $\X$ is weak-$*$ dense in the closed unit ball of $\fL$.
Then,
$
\E_\fW(\X)=\Hen_{\fW}(\X).
$
\end{proposition}
\begin{proof}
Let $\phi\in\Hen_\fW(\X)$. Given $f\in \L$, by assumption we may find a bounded net $(a_i)$ in $\X$ converging to $f$ in the weak-$*$ topology of $\fW$. We claim that $(\phi(a_i))$ is a convergent net. Since $(a_i)$ is bounded, to see this it suffices to show that $(\phi(a_i))$ only has one cluster point. Let $(\phi(b_j))_{j\in J}$ and $(\phi(c_k))_{k\in K}$ be two convergent subnets of the net $(\phi(a_i))$. Let $\Lambda=J\times K$ be equipped with the partial order such that $(j,k)\prec (j',k')$ if and only if $j\prec j'$ and $k\prec k'$. Let $d_\lambda=b_j-c_k$ if $\lambda=(j,k)$. Then, we see that $(d_\lambda)$ is a bounded net in $\X$ converging to $0$ in the weak-$*$ topology of $\fW$. The fact that $\phi$ is Henkin relative to $\fW$ implies that $(\phi(d_\lambda))$ converges to $0$, whence
\[
\lim_{j\in J}\phi(b_j)=\lim_{k\in K}\phi(c_k).
\]
This proves the claim that $(\phi(a_i))$ is convergent. 
Next, assume that $(a'_\mu)_{\mu}$ is another bounded net in $\X$ converging to $f$ in the weak-$*$ topology of $\fW$. The net $(a_i-a'_\mu)_{(i,\mu)}$ converges to $0$ in the weak-$*$ topology of $\fW$, and hence, as above,
\[
\lim_{i\in I}\phi(a_i)=\lim_{\mu}\phi(a'_\mu).
\]
Thus, we obtain a well-defined linear map $\Phi:\L\to \bC$ such that
\[
\Phi(f)=\lim_i \phi(a_i)
\]
for every bounded net $(a_i)$ in $\X$ converging to $f$ in the weak-$*$ topology of $\fW$. It is clear that $\Phi$ extends $\phi$. Moreover, because the closed unit ball of $\X$ is weak-$*$ dense in that of $\L$ by assumption, we conclude that $\|\Phi\|\leq \|\phi\|$. 

We now show that  $\Phi$ is weak-$*$ continuous. By the Krein--Smulyan theorem, it suffices to show that the restriction of $\Phi$ to the unit ball of $\L$ is weak-$*$ continuous. To see this, fix $f\in \L$ with $\|f\|\leq1$ and let $(f_\alpha)$ be a net in $\L$ with $\|f_\alpha\|\leq1$ that converges to $f$ in the weak-$*$ topology of $\fW$. We must show that $\Phi(f)$ is the only cluster point of the net $(\Phi(f_\alpha))$. Let $z\in \bC$ be such a cluster point.  A standard argument shows that we can find a bounded net $(a_i)$ in $\X$ converging to $f$ in the weak-$*$ topology of $\fW$ and such that 
 $
 \lim_i\phi(a_i)=z.
 $
By definition of $\Phi$, we find
 \[
 \Phi(f)= \lim_i\phi(a_i)=z
 \]
 as desired.
Finally, because $\L$ is weak-$*$ closed, a standard duality argument shows that $\Phi$ can be extended to a weak-$*$ continuous functional  on $\fW$ so indeed $\phi\in \E_\fW(\X)$.
\end{proof}

%
%
%

\subsection{Analytic and coanalytic triples}\label{SS:anal}

For the  purpose of characterizing the functionals on a $\rC^*$-algebra that restrict to be Henkin on a subalgebra, we introduce some terminology. Let $\fW$ be a von Neumann algebra, let $\fT\subset \fW$ be a unital $\rC^*$-subalgebra and let $\A\subset \fT$ be a unital norm-closed subalgebra. Let $\Delta\subset \fT^*$ be the norm closure of the convex hull of $\{|\phi|:\phi\in \A^\perp,\|\phi\|=1\}$. 
We say that the triple $\A\subset \fT\subset \fW$ is \emph{analytic} if a bounded net $(b_i)$ in $\A$ converges to $0$ in the weak-$*$ topology of $\fW$ whenever 
\[
\lim_i \alpha(b_i)=0, \quad \alpha\in \AC(\Delta).
\]
We now identify a set of conditions that are sufficient to guarantee the analyticity of a triple.

\begin{proposition}\label{P:anal}
Let $\H$ be a Hilbert space.
Let $\fT\subset B(\H)$ be a unital $\rC^*$-subalgebra and let $\A\subset \fT$ be a unital norm-closed subalgebra. Assume that $\fT$ contains the ideal of compact operators on $\H$. Assume also that there is a dense set $\Gamma\subset \H$ consisting of non-cyclic vectors for $\{a^*:a\in \A\}$. 
Then, for every pair of vectors $h,k\in \H$, the functional
\[
t\mapsto \langle th,k\rangle, \quad t\in \fT
\]
lies in $\AC(\Delta)$. In particular, the triple $\A\subset \fT\subset B(\H)$ is analytic.
\end{proposition}
\begin{proof}
Fix unit vectors $h,k\in \H$ and define $\psi:\fT\to \bC$ as 
\[
\psi(t)=\langle th,k\rangle, \quad t\in \fT.
\]
Recall that $\AC(\Delta)$ is a norm-closed subspace by Lemma \ref{L:ACclosed}. Thus, since $\Gamma$ is dense in $\H$, it suffices to establish that $\psi\in \AC(\Delta)$ under the additional assumption that $k\in \Gamma$. Therefore, we may find  a unit vector $x\in \H$ with $\langle ax,k\rangle=0$ for every $a\in \A$. Define functionals $\phi$ and $\omega$ on $\fT$ as
\[
\phi(t)=\langle tx ,k \rangle \qand \omega(t)=\langle tk,k\rangle
\]
for every $ t\in \fT.$ By choice of $x$, we see that $\phi\in \A^\perp$. Because $\fT$ is assumed to contain the compact operators, we see that $\|\phi\|=1$, and that there is a partial isometry $v\in \fT$ such that $vk=x$ and $v^*v$ is the projection onto $\bC k$. Then, we see that $\phi=\omega(\cdot v)$, and it readily follows from Lemma \ref{L:takunique} that $\omega=|\phi|$. Therefore,  $\omega\in \Delta$. We may now choose another partial isometry $w\in \fT$ with the property that $wk=h$. Then, $\psi=\omega(\cdot w)$, so that $\psi\ll\omega$ by Lemma \ref{L:ACmodule} and $\psi\in \AC(\Delta)$ as desired.
\end{proof}

We will also make use of the following variant. Let $\Delta_*\subset \fT^*$ denote the norm closure of the convex hull of $\{|\phi^\dagger|:\phi\in \A^\perp,\|\phi\|=1\}$. 
The triple $\A\subset \fT\subset\fW$ is said to be \emph{coanalytic} if a bounded net $(b_i)$ in $\A$ converges to $0$ in the weak-$*$ topology of $\fW$ whenever 
\[
\lim_i \alpha(b^*_i)=0, \quad \alpha\in \AC(\Delta_*).
\]
As above, we identify conditions that imply coanalyticity. 

\begin{proposition}\label{P:coanal}
Let $\H$ be a Hilbert space.
Let $\fT\subset B(\H)$ be a unital $\rC^*$-subalgebra and let $\A\subset \fT$ be a norm-closed unital subalgebra. Assume that $\fT$ contains the ideal of compact operators on $\H$. Assume also that there is a subspace $\Gamma\subset \H$ consisting of non-cyclic vectors for $\A$ with the property that for every $h\in \H\ominus \Gamma$, there is $h'\in \Gamma$ such that
\[
\langle ah,h \rangle-\langle ah',h'\rangle=0, \quad a\in \A.
\]
Then, for every pair of vectors $h,k\in \H$, the functional
\[
t\mapsto \langle th,k\rangle, \quad t\in \fT
\]
lies in $\AC(\Delta_*)$. In particular, the triple $\A\subset \fT\subset B(\H)$ is coanalytic.

\end{proposition}
\begin{proof}
Fix unit vectors $h,k\in \H$ and define a linear functional $\psi:\fT\to \bC$  as
\[
\psi(t)=\langle th,k\rangle, \quad t\in \fT.
\]
Since $\AC(\Delta_*)$ is a subspace of $\fT^*$ by Lemma \ref{L:ACclosed}, it is sufficient to establish that $\psi\in \AC(\Delta_*)$ if $k\in \Gamma$ or $k\in \H\ominus \Gamma$.

First, assume that $k\in \Gamma$, so that there is a unit vector $x\in \H$ with $\langle ak,x\rangle=0$ for every $a\in \A$. Define bounded linear functionals $\phi$ and $\omega$ on $\fT$ as
\[
\phi(t)=\langle tk ,x \rangle \qand \omega(t)=\langle tk,k\rangle
\]
for every $ t\in \fT.$  We see that $\phi\in \A^\perp$ by choice of $x$. Further,
\[
\phi^\dagger(t)=\langle tx, k\rangle, \quad t\in \fT.
\]
Since $\fT$ contains the compact operators, we see that $\|\phi\|=1$, and that there is a  partial isometry $v\in \fT$ such that $vk=x$ and $v^*v$ is the projection onto $\bC k$. Hence, $\phi^\dagger=\omega(\cdot v)$. Using Lemma \ref{L:takunique}, we infer that $\omega=|\phi^\dagger|\in \Delta_*$. Similarly, there is a partial isometry $w\in \fT$ with the property that $wk=h$. Then, we find $\psi=\omega( \cdot w)$, so that $\psi\ll \omega$ by Lemma \ref{L:ACmodule}. 
In particular, we have $\psi\in \AC(\Delta_*)$.

It remains only to deal with the case where $k\in \H\ominus \Gamma$. By choice of $\Gamma$, there is $k'\in \Gamma$ such that
\[
\langle ak,k \rangle-\langle ak',k'\rangle=0, \quad a\in \A.
\]
Define positive linear functionals $\omega_0$ and $\omega_1$ on $\fT$ as
\[
\omega_0(t)=\langle tk, k\rangle \qand \omega_1(t)=\langle tk', k'\rangle
\]
for $t\in \fT$. We have $\omega_0-\omega_1\in \A^\perp$, and because $\omega_0-\omega_1$ is self-adjoint, this implies that $\omega_0-\omega_1\in \AC(\Delta_*)$. On the other hand, $\omega_1\in \AC(\Delta_*)$ by the previous paragraph, so that $\omega_0\in \AC(\Delta_*)$ as well by Lemma \ref{L:ACclosed}. As above, we may find $r\in \fT$ such that $rk=h$ so that $\psi=\omega_0(\cdot r)$ and $\psi\ll\omega_0$ by Lemma \ref{L:ACmodule}. Hence, $\psi\in \AC(\Delta_*)$.
\end{proof}

We remark that, if $\frak{T}$ is commutative, then $\Delta=\Delta_*$. 
It then follows from Theorem \ref{T:RN} in this commutative setting that a triple is analytic if and only if it is coanalytic.

\subsection{The characterization}\label{SS:nccr}

In this subsection, we obtain our main results, which generalize the Cole--Range theorem for general $\rC^*$-algebras.

\begin{theorem}\label{T:nccr}
Let $\fW$ be a von Neumann algebra, let $\fT\subset \fW$ be a unital $\rC^*$-subalgebra and let $\A\subset \fT$ be a unital norm-closed subalgebra. Let $\Delta$ and $\Delta_*$ denote the norm closure in $\fT^*$ of the convex hull of
\[
\{|\phi|:\phi\in \A^\perp,\|\phi\|=1\} \qand \{|\phi^\dagger|:\phi\in \A^\perp,\|\phi\|=1\} 
\] 
respectively. Moreover, we let
\[
\Band=\{\phi\in \fT^*:\phi|_\A\in \Hen_{\fW}(\A)\}.
\]
Then, the following statements hold.
\begin{enumerate}[{\rm (i)}]
\item The triple $\A\subset \fT\subset \fW$ is analytic if and only if
$
\Band\subset \AC(\Delta).
$

\item  The triple $\A\subset \fT\subset \fW$ is coanalytic if and only if 
$
\Band^\dagger\subset \AC(\Delta_*).
$

\end{enumerate}

\end{theorem}
\begin{proof}
(i) Assume first that $\Band\subset \AC(\Delta).$ Let $(b_i)$ be a bounded net in $\A$ such that
\[
\lim_i \alpha(b_i)=0, \quad \alpha\in \AC(\Delta).
\]
Let $\omega$ be a weak-$*$ continuous functional on $\fW$. We find $\omega|_{\fT}\in \Band$, so by assumption we see that
\[
\lim_i \omega(b_i)=0.
\]  
We conclude that $(b_i)$ converges to $0$ in the weak-$*$ topology of $\fW$, so the triple $\A\subset \fT\subset \fW$ is analytic.

Conversely, assume that the triple $\A\subset \fT\subset \fW$ is analytic. 
Let $\phi\in \Band$, and set $\fr=\vee_{\delta\in\Delta}\fs_{\delta}$ and $\sigma=\phi((I-\frk{r})\cdot)$. 
By virtue of Theorem \ref{T:ncRiesz} we have $\phi-\sigma\in \AC(\Delta)$ and $\sigma\in \SG(\Delta)$. 
Fix $t\in \fT$. For $\delta\in \Delta$, we clearly have that $\widehat\delta(I-\fr)=0$, and in turn we find
\begin{equation}\label{Eq:Aperpperp}
\widehat\alpha((I-\fr)t)=0
\end{equation}
for every $\alpha\in \AC(\Delta)$ by Lemma \ref{L:statevanish}. By definition, we have $\A^\perp\subset \AC(\Delta)$, so we conclude from \eqref{Eq:Aperpperp} that $(I-\fr)t\in \A^{\perp\perp}$. Seeing as $\A^{\perp\perp}$ is isometrically and weak-$*$ homeomorphically isomorphic to $\A^{**}$, we may apply Goldstine's theorem to obtain a bounded net $(b_i)$ in $\A$ converging to $(I-\fr)t$ in the weak-$*$ topology of $\fT^{**}$. Given $\alpha\in\AC(\Delta)$, we see that
\[
\lim_i \alpha(b_i)=\widehat\alpha((I-\fr)t)=0
\]
by virtue of \eqref{Eq:Aperpperp}.  Since the triple $\A\subset \fT\subset \fW$ is analytic, we infer that $(b_i)$ converges to $0$ in the weak-$*$ topology of $\fW$. Seeing as $\phi\in \Band$, this forces 
\[
\sigma(t)=\widehat\phi((I-\fr)t)=\lim_i \phi(b_i)=0.
\]
Because $t\in \fT$ was arbitrary, it follows that $\sigma=0$ so indeed $\phi\in \AC(\Delta)$.

(ii) The proof is very similar to the previous one, but we provide the details for completeness. Assume first that $\Band^\dagger\subset \AC(\Delta_*)$.
Let $(b_i)$ be a bounded net in $\A$ such that
\[
\lim_i \alpha(b^*_i)=0, \quad \alpha\in \AC(\Delta_*).
\]
Let $\omega$ be a weak-$*$ continuous functional on $\fW$. We find $\omega|_\fT\in \Band$, so by assumption $\omega^\dagger|_\fT \in \AC(\Delta_*)$ and we see that
\[
\lim_i \omega(b_i)=\lim_i\ol{ \omega^\dagger(b^*_i)}=0.
\]  
We conclude that $(b_i)$ converges to $0$ in the weak-$*$ topology of $\fW$, and hence the triple $\A\subset \fT\subset \fW$ is coanalytic.

Conversely, assume that the triple $\A\subset \fT\subset \fW$ is coanalytic.
Let $\phi\in \Band$, and set $\frk{r}'=\vee_{\delta\in\Delta_*}\frk{s}_\delta$ and $\sigma=\phi^{\dagger}( (I-\frk{r}')\cdot)$. 
By virtue of Theorem \ref{T:ncRiesz}, we have $\phi^\dagger-\sigma\in \AC(\Delta_*)$ and $\sigma\in \SG(\Delta_*)$. 
Fix $t\in \fT$. For $\delta\in \Delta_*$, we clearly have that $\widehat\delta(I-\fr')=0$, and in turn we find
\begin{equation}\label{Eq:Aperpperp*}
\widehat\alpha((I-\fr')t)=0
\end{equation}
for every $\alpha\in \AC(\Delta_*)$ by Lemma \ref{L:statevanish}. By definition, we see that
\[
\{\psi^\dagger: \psi\in \A^\perp\}\subset \AC(\Delta_*)
\]
so that
\[
\widehat\psi(t^*(I-\fr'))=\ol{\widehat\psi^\dagger((I-\fr')t)}=0, \quad\psi\in \A^\perp
\]
by \eqref{Eq:Aperpperp*}. We thus infer that $t^*(I-\fr')\in \A^{\perp\perp}$. Seeing as $\A^{\perp\perp}$ is isometrically and weak-$*$ homeomorphically isomorphic to $\A^{**}$, we may apply Goldstine's theorem to obtain a bounded net $(b_i)$ in $\A$ converging to $t^*(I-\fr')$ in the weak-$*$ topology of $\fT^{**}$. Given $\alpha\in\AC(\Delta_*)$, we see that
\[
\lim_i \alpha(b^*_i)=\widehat\alpha((I-\fr')t)=0
\]
by virtue of \eqref{Eq:Aperpperp*} again.  Since the triple $\A\subset \fT\subset \fW$ is coanalytic, we infer that $(b_i)$ converges to $0$ in the weak-$*$ topology of $\fW$. Seeing as $\phi\in \Band$, this forces 
\[
\sigma(t)=\widehat\phi^\dagger((I-\fr')t)=\lim_i \phi^\dagger(b^*_i)=\lim_i \ol{\phi(b_i)}=0.
\]
Because $t\in \fT$ was arbitrary, it follows that $\sigma=0$ so indeed $\phi^\dagger\in \AC(\Delta_*)$.
\end{proof}

We now aim to sharpen the previous theorem by establishing the reverse inclusions. The next result identifies the obstruction to doing so, based on the following concept. Given a unital $\rC^*$-algebra $\fT$, a subset $\Band\subset \fT^*$ is called a \emph{left band} if given $\phi\in \fT^*$ and $\beta\in \Band$ with $\phi\ll\beta$, we must necessarily have that $\phi\in \Band$. Further, we say that $\Band$ is a \emph{right band} if given $\phi\in \fT^*$ and $\beta\in \Band$ with $\phi^\dagger\ll\beta^\dagger$, we must necessarily have that $\phi\in \Band$. It is easily verified then that  $\Band$ is a right band if and only if $\Band^\dagger$ is a left band. We can now prove another one of our main results.

\begin{theorem}\label{T:Henband}
Let $\fW$ be a von Neumann algebra, let $\fT\subset \fW$ be a unital $\rC^*$-subalgebra and let $\A\subset \fT$ be a unital norm-closed subalgebra. Let $\Delta$ and $\Delta_*$ denote the norm closure in $\fT^*$ of the convex hull of
\[
\{|\phi|:\phi\in \A^\perp,\|\phi\|=1\} \qand \{|\phi^\dagger|:\phi\in \A^\perp,\|\phi\|=1\}
\] 
respectively.  Moreover, we let
\[
\Band=\{\phi\in \fT^*:\phi|_\A\in \Hen_{\fW}(\A)\}.
\]
Then, the following statements hold.
\begin{enumerate}[{\rm (i)}]
\item Assume that the triple $\A\subset \fT\subset \fW$ is analytic. Then,  $\Band=\AC(\Delta)$  if and only if $\Band$ is a left band.

\item Assume that the triple $\A\subset \fT\subset \fW$ is coanalytic. Then,   $\Band^\dagger=\AC(\Delta_*)$
 if and only if $\Band$ is a right band.
 
 \item Assume that the triple $\A\subset \fT\subset \fW$ is both analytic and coanalytic. Then,   $\Band=\Band^\dagger=\AC(\Delta)=\AC(\Delta_*)$
 if and only if $\Band$ is both a left band and a right band.
 
\end{enumerate}
\end{theorem}
\begin{proof}

(i) It is trivial that the equality  $\Band=\AC(\Delta)$
implies that $\Band$ is a left band.  Assume conversely that $\Band$ is a left band. Note that $\Band\subset \AC(\Delta)$ by Theorem \ref{T:nccr}, so we aim to establish the reverse inclusion. Let $\phi\in \A^\perp$. Trivially, we have that $\phi\in \Band$. But $|\phi|\ll\phi$, so by our assumption we infer that $|\phi|\in \Band$. Since $\Hen_\fW(\A)$ is a norm-closed subspace of $\A^*$, we infer that $\Delta\subset \Band$. Using once again the fact that $\Band$ is a left band yields $\AC(\Delta)\subset \Band$.

(ii) The proof is nearly identical to the one above. If $\Band^\dagger=\AC(\Delta_*)$, then it is trivial that $\Band^\dagger$ is a left band, whence $\Band$ is a right band. Assume conversely that $\Band$ is a right band, or equivalently that $\Band^\dagger$ is a left band. Note that $\Band^\dagger\subset \AC(\Delta_*)$ by Theorem \ref{T:nccr}, so we aim to establish the reverse inclusion.  Let $\phi\in \A^\perp$. Trivially, we have that $\phi\in \Band$ and thus $\phi^\dagger\in \Band^\dagger$. But $|\phi^\dagger|\ll\phi^\dagger$, so by our assumption we infer that $|\phi^\dagger|\in \Band^\dagger$. Since $\Hen_\fW(\A)$ is a norm-closed subspace of $\A^*$, we infer that $\Delta_*\subset \Band^\dagger$. Using again that $\Band^\dagger$ is a left band yields $\AC(\Delta_*)\subset \Band^\dagger$. 

(iii) Assume that $\Band$ is both a left band and a right band. Combining (i) and (ii), we find $\Band=\AC(\Delta)=\AC(\Delta_*)^\dagger$. 
Let $\beta\in \Band$. There is a state $\delta\in \Delta$ such that $\beta\ll\delta$. Equivalently, we have that $(\beta^\dagger)^\dagger\ll \delta^\dagger$, since $\delta$ is self-adjoint. Using now that $\delta\in \Delta\subset \Band$ and that $\Band$ is a right band, we find $\beta^\dagger\in \Band$. We conclude that $\Band=\Band^\dagger$, so indeed   $\Band=\Band^\dagger=\AC(\Delta)=\AC(\Delta_*)$. The converse is clear.
\end{proof}

In the next section, we will see that there are natural examples where only one of the statements from Theorems \ref{T:nccr} and \ref{T:Henband} holds. Thus, it is important to distinguish the left and right versions of our objects of interest in the foregoing discussion. This is, of course, a reflection of the fact that the $\rC^*$-algebra $\fT$ may not be commutative.

\section{Examples}\label{S:Ex}

In this section, we apply the general tools developed in Section \ref{S:ncHenkin} to concrete examples of importance in multivariate operator theory.
Throughout, we retain the notation used above: given a triple $\A\subset \fT\subset \fW$, we let  $\Delta$ and $\Delta_*$ denote the norm closure in $\fT^*$ of the convex hull of
\[
\{|\phi|:\phi\in \A^\perp,\|\phi\|=1\} \qand \{|\phi^\dagger|:\phi\in \A^\perp,\|\phi\|=1\}
\] 
respectively.  Moreover, we let
\[
\Band=\{\phi\in \fT^*:\phi|_\A\in \Hen_{\fW}(\A)\}.
\]

\subsection{The ball algebra}
Let $d\geq 1$ be an integer. Let $\bB_d\subset \bC^d$ denote the open unit ball, and let $\bS_d$ denote its topological boundary, the unit sphere. Recall that the ball algebra $\AB\subset \rC(\bS_d)$ is the norm closure of the polynomials.  Let $\sigma$ denote the unique rotation invariant Borel probability measure on $\bS_d$. In this context, Cauchy's formula  \cite[Paragraph 3.2.4]{rudin2008}  says that given $f\in \AB$ and $\lambda\in \bB_d$, we have 
\begin{equation}\label{Eq:Cauchy}
f(\lambda)=\int_{\bS_d}\frac{f(\zeta)}{(1-\langle \zeta,\lambda\rangle)^d}d\sigma(\zeta).
\end{equation}
Let $H^\infty(\bB_d)$ denote the algebra of bounded holomorphic functions on $\bB_d$, equipped with the supremum norm over $\bB_d$. A function in $\HB$ has $[\sigma]$-a.e defined radial boundary values \cite[Section 5.6]{rudin2008}. An application of \eqref{Eq:Cauchy} together with the dominated convergence theorem then reveals that $H^\infty(\bB_d)$ can be embedded isometrically inside of $L^\infty(\bS_d,\sigma)$. Furthermore, standard properties of the Poisson kernel \cite[Section 3.3]{rudin2008} imply that  the corresponding copy of $H^\infty(\bB_d)$ is a weak-$*$ closed subalgebra of  $L^\infty(\bS_d,\sigma)$. In this fashion, $\HB$ is endowed with a weak-$*$ topology. It is a routine consequence of \eqref{Eq:Cauchy} that a bounded net $(f_i)$ in $H^\infty(\bB_d)$ converges in this weak-$*$ topology  precisely when it converges pointwise on $\bB_d$. 

Given $f\in \HB$ and $r>0$, if we define $f_r:\bB_d\to \bC$ as
\[
f_r(z)=f(rz), \quad z\in \bB_d
\]
then $f_r\in \AB$ and $\|f_r\|\leq \|f\|$. It is readily seen that the net $(f_r)$ converges to $f$ pointwise on $\bB_d$, and hence in the weak-$*$ topology of $H^\infty(\bB_d)$. We conclude that the closed unit ball of $\AB$ is weak-$*$ dense in that of $\HB$. In turn, Proposition \ref{P:Esc} implies that
\[
\Hen_{L^\infty(\bS_d,\sigma)}(\AB)=\E_{L^\infty(\bS_d,\sigma)}(\AB)
\]
which is equivalent to the classical Valskii decomposition \cite[Theorem 9.2.1]{rudin2008}.

Next, we verify that the triple
\[
\AB\subset \rC(\bS_d)\subset L^\infty(\bS_d,\sigma)
\] 
is both analytic and coanalytic. The general tools developed in Subsection \ref{SS:anal} do not apply here, so we provide a direct argument.

\begin{lemma}\label{L:ABanal}
The triple $\AB\subset \rC(\bS_d)\subset L^\infty(\bS_d,\sigma)$ is both analytic and coanalytic.
\end{lemma}
\begin{proof}
	Because $\rC(\mathbb{S}_d)$ is commutative, it suffices to verify analyticity of the triple. 
 For $1\leq k\leq d$ we define a non-zero bounded linear functional $\psi_k: \rC(\bS_d)\to \bC$ as
\[
\psi_k(f)=\int_{\bS_d}fz_k d\sigma, \quad f\in \rC(\bS_d).
\]
By the rotation invariance of the measure $\sigma$, there is a number $c\geq 1$ such that $\|\psi_k\|=1/c$ for all $1\leq k\leq d$. 
Put $\psi'_k=c\psi_k$. 
Note then, for each $1\leq k\leq d$ and each $f\in \rC(\bS_d)$, that
\[
|\psi'_k|(f)=c \int_{\bS_d}f|z_k|d\sigma, 
\]
by virtue of Lemma \ref{L:absvalmeasure}.
It follows from \eqref{Eq:Cauchy} that $\psi_k\in \rA(\bB_d)^\perp$, so that $|\psi'_k|\in \Delta$. 
Let 
$
\delta= \frac{1}{d}\sum_{k=1}^d |\psi'_k|.
$
Then, $\delta\in \Delta$ 
and 
\[
\delta(f)=\int_{\bS_d} f\cdot\left( \frac{c}{d}\sum_{k=1}^d  |z_k|\right)d\sigma,\quad f\in\rC(\bS_d).
\]
On the sphere $\bS_d$, we have
\[
\frac{c}{d} \leq \frac{c}{d}\sum_{k=1}^d  |z_k|\leq \frac{c}{\sqrt{d}}
\]
whence the measure $\left( \frac{c}{d}\sum_{k=1}^d  |z_k|\right)\sigma$ is mutually absolutely continuous with $\sigma$. By Lemma \ref{L:ACmeasure}, this implies that the state $\rho: \rC(\bS_d)\to\bC$ defined as
\[
\rho(f)=\int_{\bS_d}f d\sigma, \quad f\in \rC(\bS_d)
\]
is absolutely continuous with respect to $\delta$, so that $\rho\in \AC(\Delta)$. 
For each $\lambda\in \bB_d$, we now define $\alpha_\lambda:  \rC(\bS_d)\to \bC$ as
\[
\alpha_\lambda(f)=\int_{\bS_d}\frac{f(\zeta)}{(1-\langle\zeta,\lambda \rangle)^d}d\sigma(\zeta), \quad f\in \rC(\bS_d).
\]
By \eqref{Eq:Cauchy}, we see that
\[
\alpha_\lambda(f)=f(\lambda), \quad f\in \AB.
\]
On the other hand, another application of Lemma \ref{L:ACmeasure}  reveals that, for each $\lambda\in \bB_d$, we have $\alpha_\lambda\ll\rho$, 
so that $\alpha_\lambda\in \AC(\Delta)$. 
Thus, if  a bounded net  $(f_i)$ in $\rA(\bB_d)$ satisfies
\[
\lim_i \alpha(f_i)=0, \quad \alpha\in \AC(\Delta) 
\]
necessarily the net $(f_i)$ converges pointwise to $0$ on $\bB_d$. Equivalently, the net $(f_i)$ converges to $0$ in the weak-$*$ topology of $L^\infty(\bS_d,\sigma)$, as discussed before the lemma. We conclude that the required triple is indeed analytic.
\end{proof}

Next, we move on to the identification of the set $\Band$. We start with a classical result.

\begin{theorem}[Henkin]\label{T:classHenkin}
The set  $\Band$ is both a left band and a right band.
\end{theorem}
\begin{proof}
This is an immediate consequence of Lemma \ref{L:ACmeasure} and of Henkin's theorem \cite[Theorem 9.3.1]{rudin2008}.
\end{proof}

We are now in a position to apply our findings from Section \ref{S:ncHenkin}.

\begin{corollary}\label{C:crAB}
We have
$
\Band=\Band^\dagger=\AC(\Delta)=\AC(\Delta_*).
$
\end{corollary}
\begin{proof}
This follows upon combining Theorem \ref{T:Henband} with Lemma \ref{L:ABanal} and Theorem \ref{T:classHenkin}. 
\end{proof}

It is interesting to contrast the previous result with the following classical fact, which combines work of Glicksberg--K\"onig--Seever \cite{glicksberg1967}, Henkin \cite{henkin1968}, Valskii   \cite{valskii1971} and Cole--Range  \cite{cole1972}. Let  $\mathscr{R}_0$ be the set of states $\mu:\rC(\bS_d)\to\bC$ with the property that
\[
\mu(f)=f(0), \quad f\in \AB.
\]

\begin{theorem}\label{T:crAB}
We have that $\Band=\Band^\dagger=\AC(\mathscr{R}_0).$
\end{theorem}
\begin{proof}
This follows from Lemma \ref{L:ACmeasure} and Corollary \ref{C:crAB} along with \cite[Theorem 9.6.1]{rudin2008}.
\end{proof}

In light of Corollary \ref{C:crAB} and Theorem \ref{T:crAB}, we infer that
\[
\AC(\Delta)=\AC(\Delta_*)=\AC(\mathscr{R}_0).
\]
In turn, this can be translated into purely measure theoretic terms via Lemma \ref{L:ACmeasure}. We do not know of a direct elementary proof of this fact.

\subsection{Algebras of multipliers on the ball}\label{SS:AH}
Let $d\geq 1$ be an integer and let $\H$ be a reproducing kernel Hilbert space on $\bB_d$, with kernel function $k:\bB_d\times \bB_d\to \bC$. We always assume that there exists a sequence $(a_n)$ of positive numbers such that $a_0=1$ and
\[
k(z,w)=\sum_{n=0}^\infty a_n \langle z,w \rangle^n,\quad z,w\in\bB_d.
\]
In addition, we make the standing assumption that
\[
\lim_{n\to\infty}\frac{a_n}{a_{n+1}}=1.
\]
Often, these conditions are summarized by saying that $\H$ is a \emph{regular unitarily invariant space}. Examples of such spaces include the Hardy and Bergman spaces on the ball, the Dirichlet space on the disc, and the Drury--Arveson space on the ball  \cite{hartz2017isom}.

It follows from \cite[Corollary 4.4]{GHX04} that for each polynomial $p\in \bC[z_1,\ldots,z_d]$, the corresponding multiplication operator $M_p:\H\to\H$ is bounded.  Accordingly, we denote by $\A(\H)\subset B(\H)$ the closure of the polynomial multipliers. We also put $\fT(\H)=\rC^*(\A(\H))\subset B(\H)$. It follows from \cite[Theorem 4.6]{GHX04} that $\fT(\H)$ always contains the ideal  $\fK$ of compact operators on $\H$, and that there is a surjective $*$-homomorphism $q:\fT(\H)\to \rC(\bS_d)$ with kernel $\fK$ such that 
\[
q(M_f)=f, \quad f\in \A(\H).
\]
We will require the following standard facts numerous times below.

\begin{lemma}\label{L:fz}
Let $\H$ be a regular unitarily invariant space on $\bB_d$. Then, there is a central projection $\fz\in \fT(\H)^{**}$ with the following properties.
\begin{enumerate}[{\rm (i)}]
\item We have $\fK^{\perp\perp}=\fT(\H)^{**}(I-\fz)$ and $\rC(\bS_d)^{**}\cong \fT(\H)^{**}\fz$.
\item Let $\phi:B(\H)\to\bC$ be a bounded linear functional. Then, $\phi$ is weak-$*$ continuous if and only if $\phi=\widehat\phi(\cdot (I-\fz))$. 
\item There is a weak-$*$ homeomorphic $*$-isomorphism $\Omega: \fT(\H)^{**}(I-\fz)\to B(\H)$ such that
\[
\Omega(t(I-\fz))=t, \quad t\in \fT(\H).
\]
\end{enumerate}
\end{lemma}
\begin{proof}
Consider the weak-$*$ continuous surjective $*$-homomorphism $q^{**}:\fT(\H)^{**}\to \rC(\bS_d)^{**}$. We find $\ker q^{**}=\fK^{\perp\perp}$, which is a weak-$*$ closed two-sided ideal in $B(\H)^{**}$, so by \cite[Proposition 1.10.5]{sakai1971} there is a central projection $\fz\in B(\H)^{**}$ with the property that 
\[
\fK^{\perp\perp}=B(\H)^{**}(I-\fz)=\fT(\H)^{**}(I-\fz).
\]
We find
\[
\rC(\bS_d)^{**}=q^{**}(\fT(\H)^{**})\cong \fT(\H)^{**}/\ker q^{**}=\fT(\H)^{**}\fz.
\]
This establishes statement (i). Statement (ii) follows from the discussion preceding  \cite[Lemma 2.5]{CTh2020fdim}. 

To see (iii), let $\Phi:B(\H)\to B(\H)^{**}(I-\fz)$ be the $*$-homomorphism defined as
\[
\Phi(x)=x(I-\fz),\quad x\in B(\H).
\]
It follows readily from (ii) that $\Phi$ is weak-$*$ continuous and isometric. Thus, $\Phi$ is a weak-$*$ homeomorphism onto $\Phi(B(\H))$ \cite[Theorem A.2.5]{BLM2004}. We conclude  that \[\Phi(B(\H))= B(\H)^{**}(I-\fz)=\fT(\H)^{**}(I-\fz)\] so we may then simply take $\Omega=\Phi^{-1}$.
\end{proof}

We now analyze the triple
\[
\A(\H)\subset \fT(\H)\subset B(\H)
\]
similarly to what was done for the ball algebra in the previous subsection. Before proceeding, we note that by choosing $\H$ to be the Hardy space on the ball, then $\A(\H)$ can be identified completely isometrically with $\AB$. However, the corresponding $\rC^*$-algebra $\fT(\H)$ is not commutative, so that even in this case the investigation that follows is of a rather different nature than what was done previously.

First, we note that \cite[Lemma 3.1]{BHM2018} yields
\begin{equation}\label{Eq:Henkinext}
\Hen_{B(\H)}(\A(\H))=\E_{B(\H)}(\A(\H)).
\end{equation}
Next, we investigate the analyticity and coanalyticity of our triple of interest.

\begin{lemma}\label{L:analAH}
Let $\H$ be a regular unitarily invariant space on $\bB_d$. Then, the triple
$
\A(\H)\subset \fT(\H)\subset B(\H)
$
is both analytic and coanalytic.
\end{lemma}
\begin{proof}
For each $\lambda\in \bB_d$, let $k_\lambda\in \H$ denote the corresponding kernel vector. Consider the set
\[
\Gamma_1=\spn\{k_\lambda:\lambda\in \bB_d\}
\]
which is dense in $\H$. We know that that $M_f^* k_\lambda=\ol{f(\lambda)}k_\lambda$ for every $f\in \A(\H)$. In particular, this implies that if $h\in \Gamma_1$, then the subspace $\{M_f^* h:f\in \A(\H)\}$ is finite-dimensional and hence is a proper closed subset of $\H$. 
Next, let $\Gamma_2=(\bC k_0)^\perp$. This is a closed and proper subspace of $\H$ which is invariant under $\A(\H)$, so that no vector in $\Gamma_2$ can be cyclic for $\A(\H)$. Furthermore, we see that $z_1\in \Gamma_2$ and by \cite[Proposition 4.3]{GHX04} we obtain
\[
\langle fk_0, k_0\rangle=a_1 \langle fz_1,z_1\rangle, \quad f\in \A(\H).
\]
We can now apply Propositions \ref{P:anal} and \ref{P:coanal} to see that the triple
\[
\A(\H)\subset \fT(\H)\subset B(\H)
\]
is both analytic and coanalytic.
\end{proof}

Even though the $\rC^*$-algebra $\fT(\H)$ is non-commutative, we still obtain an analogue of Theorem \ref{T:classHenkin}.

\begin{theorem}\label{T:HenAH}
Let $\H$ be a regular unitarily invariant space on $\bB_d$. Then, the set  $\Band$
is both a left band and a right band.
\end{theorem}
\begin{proof}
Fix bounded linear functionals $\phi$ and $\psi$ on $\fT(\H)$ and suppose $\psi\in\Band$. By Lemma \ref{L:fz}, there is a central projection $\fz\in B(\H)^{**}$ with the property that the functionals $\widehat\phi((I-\fz)\cdot)$ and $\widehat\psi((I-\fz)\cdot)$ both lie in $\Band$, while $\widehat\phi(\fz \cdot)$ and $\widehat\psi(\fz \cdot)$ both annihilate $\fK$. Thus, we may find bounded linear functionals $\phi'$ and $\psi'$ on $\rC(\bS_d)$ such that
\[
\widehat\phi(\fz \cdot)=\phi'\circ q \qand\widehat\psi(\fz \cdot)=\psi'\circ q.
\]
Furthermore, there are regular Borel measures $\mu$ and $\nu$ on $\bS_d$ such that
\[
\widehat \phi(\fz t)=(\phi'\circ q)(t)=\int_{\bS_d}q(t)d\mu \qand \widehat \psi(\fz t)=(\psi'\circ q)(t)=\int_{\bS_d}q(t)d\nu
\]
for every $t\in \fT(\H)$.

Assume now that $\phi\ll\psi$. Since $\fz$ is central, it follows from Lemma \ref{L:absvaluecomp} that the absolute values of $\widehat\phi(\fz \cdot)$ and $\widehat\psi(\fz\cdot)$ are given by $|\phi|(\fz \cdot)$ and $|\psi|(\fz\cdot)$ respectively.
We infer that $\widehat\phi(\fz\cdot)\ll\widehat\psi(\fz \cdot)$. By another application of Lemma \ref{L:absvaluecomp}, we obtain
\[
|\phi|(\fz \cdot)=|\phi'|\circ q^{**} \qand|\psi|(\fz \cdot)=|\psi'|\circ q^{**} 
\]
which implies that $\phi'\ll\psi'$. Then, Lemma \ref{L:ACmeasure} yields that $\mu\ll\nu$ as measures, so we may invoke \cite[Lemma 3.3]{BHM2018} to conclude that $\widehat\phi(\fz\cdot)\in \Band$. Thus, 
\[
\phi=\widehat\phi((I-\fz)\cdot)+\widehat\phi(\fz \cdot)\in\Band
\]
which shows that $\Band$ is a left band.

Finally, assume that $\phi^\dagger\ll\psi^\dagger$. As above, we find $\phi'^{\dagger}\ll \psi'^\dagger$. But $\rC(\bS_d)$ is commutative,  so Theorem \ref{T:RN}  implies that $\phi'\ll\psi'$. Arguing as in the previous paragraph, we find $\phi\in \Band$ so $\Band$ is a right band.
\end{proof}

We can now identify the set $\Band$.

\begin{corollary}\label{C:crAH}
Let $\H$ be a regular unitarily invariant space on $\bB_d$.
Then, we have
$
\Band=\Band^\dagger=\AC(\Delta)=\AC(\Delta_*).
$\end{corollary}
\begin{proof}
This follows from Theorem \ref{T:Henband} along with Lemma \ref{L:analAH} and Theorem \ref{T:HenAH}.
\end{proof}

Motivated by the corresponding analysis that was carried out for the ball algebra, it is now natural to ask whether there is a version of the previous theorem that involves the set $\mathscr{R}_0$ of states $\rho$ on $\fT(\H)$ with the property that
\[
\rho(M_f)=f(0), \quad f\in \A(\H).
\]
More precisely, we wonder whether or not $\Band=\AC(\Ro)$. As mentioned in the introduction (see Question \ref{Q:crAH}), we do not know the answer.

It is worthwhile to discuss this question further in the important case of the Drury--Arveson space $H^2_d$ on $\bB_d$. Given a regular Borel measure $\mu$ on $\bS_d$, we define a bounded linear functional $\phi_\mu$ on $\fT(H^2_d)$ by
\[
\phi_\mu(t)=\int_{\bS_d}q(t)d\mu,\quad t\in \fT.
\]
For the sake of this discussion, define $\Ro^m$ to be the set of states $\rho\in \Ro$ for which there is a regular Borel probability measure $\mu$ such that $\rho=\phi_\mu$.  When $d\geq 2$, it was shown in \cite{hartz2018henkin} that there is a regular Borel probability measure $\tau$ on $\bS_d$, singular with respect to every measure defining a state in $\Ro^m$, for which $\phi_\tau\in \Band$. By Lemma \ref{L:ACmeasure}, we infer that $\phi_\tau\in \Band\setminus \AC(\Ro^m)$. It is tempting to guess then that $\phi_\tau\in \Band \setminus \AC(\Ro)$, thereby answering Question \ref{Q:crAH}. Unfortunately, we cannot prove this at present. The issue is that membership in $\AC(\Ro)$ is a priori easier to achieve than membership in $\AC(\Ro^m)$. The next example illustrates this phenomenon.

\begin{example}\label{E:measures}
Let $\Delta\subset \fT(H^2_d)^*$ denote the norm closure of the convex hull of $\{|\psi|:\psi\in \A(H^2_d)^\perp, \|\psi\|=1\}$. Furthermore, we let $\Theta\subset \rC(\bS_d)^*$ denote the norm closure of the convex hull of  $\{|\nu|:\nu\in \AB^\perp,\|\nu\|=1\}$. With $\tau$ as above, it follows from Corollary \ref{C:crAH} that $\phi_\tau\in \AC(\Delta)$, while Theorem \ref{T:crAB} implies that $\tau \notin \AC(\Theta)$. This may seem bizarre at first glance, as given a regular Borel measure $\mu$ on $\bS_d$, we have that $\phi_\mu\in \A(H^2_d)^\perp$ if and only if $\mu\in \AB^\perp$. The issue here is that the notion of absolute continuity in $\fT(H^2_d)^*$ is much more flexible than its counterpart in $\rC(\bS_d)^*$. We substantiate this claim by identifying explicitly an element $\psi\in \A(H^2_d)^\perp$ such that $\phi_\tau\ll\psi$.

Because $\phi_\tau\in \Band$, it follows from \eqref{Eq:Henkinext} that there is a weak-$*$ continuous linear functional $\omega:B(H^2_d)\to \bC$ such that if we put $\psi=\phi_\tau- \omega|_{\fT(H^2_d)}$, then $\psi\in \A(H^2_d)^\perp$. By Lemma \ref{L:fz}, there is a central projection $\fz\in \fT(H^2_d)^{**}$ such that $\fK^{\perp\perp}=\fT(H^2_d)^{**}(I-\fz)$ and $\omega=\widehat\omega(\cdot (I-\fz))$. By definition, $\phi_\tau$ annihilates $\fK$, so that $\phi_\tau=\widehat{\phi_\tau}(\cdot \fz )$. 
For $t\in \fT(H^2_d)$ we find
\begin{align*}
\phi_\tau(t)&=\widehat{\phi_\tau}(t \fz )=\widehat\psi(\fz t)+\widehat\omega(t\fz )=\widehat\psi(t\fz ).
\end{align*}
Thus, $\phi_\tau=\widehat\psi(\cdot \fz )$, whence  $\phi_\tau$ is absolutely continuous with respect to $\psi$ by Lemma \ref{L:ACmodule}. 
\qed
\end{example}

\subsection{Popescu's noncommutative disc algebra}
Let $d\geq 1$ be an integer and let $\bF_d^+$ denote the free monoid on $d$ generators. Let $\fF^2_d$ be the full Fock space on $\bC^d$, which can be identified with $\ell^2(\bF_d^+)$. For each $w\in \bF^+_d$, we let $e_w\in \fF^2_d$ denote the characteristic function of $\{w\}$, so that $\{e_w:w\in \bF_d^+\}$ is an orthonormal basis of $\fF^2_d$. For each $1\leq j\leq d$, we let $L_j\in B(\fF^2_d)$ be the isometry determined by
\[
L_j e_w=e_{jw}, \quad w\in \bF^+_d.
\]
More generally, given a word $u\in \bF^+_d$, we let $L_u\in B(\fF^2_d)$ be the isometry determined by
\[
L_u e_w=e_{uw}, \quad w\in \bF^+_d.
\]
The norm-closed unital subalgebra $\fA_d\subset B(\fF^2_d)$ generated by $\{L_1,\ldots,L_d\}$ is Popescu's \emph{noncommutative disc algebra} \cite{popescu1989},\cite{popescu1996}. We let $\fT_d=\rC^*(\fA_d)\subset B(\fF^2_d)$, and note that this $\rC^*$-algebra contains the ideal of compact operators on $\fF^2_d$ \cite[Theorem 1.3]{popescu2006}. In this subsection, we investigate the triple
\[
\fA_d\subset \fT_d\subset B(\fF^2_d)
\]
and the associated Henkin functionals. First, we note that 
\[
\Hen_{B(\fF^2_d)}(\fA_d)=\E_{B(\fF^2_d)}(\fA_d)
\]
by virtue of Proposition \ref{P:Esc} and the proof of \cite[Theorem 1.2]{davidson1999}. Adapting the proof of Lemma \ref{L:analAH}, we can now establish analyticity and coanalyticity of the triple.

\begin{lemma}\label{L:analAd}
The triple
$
\fA_d\subset \fT_d\subset B(\fF^2_d)
$
is both analytic and coanalytic. 
\end{lemma}
\begin{proof}
Consider the set
\[
\Gamma_1=\spn\{e_w:w\in \bF^+_d\}
\]
which is obviously dense in $\fF^2_d$. Given $h\in \Gamma_1$, there is an integer $N$ large enough so that  $L_w^* h=0$ for every $w\in \bF^+_d$ with $|w|\geq N$. We conclude that $\{a^* h:a\in \fA_d\}$ is a finite-dimensional subspace and in particular is a proper closed subset of $\fF^2_d$. Next, let $\Gamma_2=(\bC e_\varnothing)^\perp$. This is a proper closed subspace of $\fF^2_d$ which is invariant under $\fA_d$, so that no vector in $\Gamma_2$ can be cyclic for $\fA_d$. Furthermore, we see that $e_1\in \Gamma_2$ and
\[
\langle ae_\varnothing, e_\varnothing\rangle=\langle ae_1,e_1\rangle, \quad a\in \fA_d.
\]
We may now apply Proposition \ref{P:anal} and \ref{P:coanal} to find that the triple
\[
\fA_d\subset \fT_d\subset B(\fF^2_d)
\]
is indeed both analytic and coanalytic. 
\end{proof}

Up to this point, it may seem as though $\fA_d$ behaves like $\A(\H)$ for our purposes. The next step will exhibit a sharp difference. Before proceeding however, we need a preliminary result.

\begin{lemma}\label{L:riesz}
Let $\phi:\fT_d\to \bC$ be a bounded linear functional annihilating $\fA_d$. Then, the following statements hold.
\begin{enumerate}[{\rm (i)}]
\item  There is a  unital $*$-homomorphism $\pi:\fT_d\to B(\H)$ and two vectors $\xi,\eta\in \H$ such that
\[
\phi(t)=\langle \pi(t)\xi,\eta\rangle \qand |\phi^\dagger|(t)=\langle \pi(t)\xi,\xi \rangle
\]
for every $t\in \fT_d$. Furthermore, if we let $\M=\ol{\pi(\fA_d)\xi}$, then there is a unitary operator $U:\fF^2_d\to \M$ such that
\[
U a U^*=\pi(a)|_\M, \quad a\in \fA_d.
\]
%

\item For each fixed $s\in \fT_d$,  we have that $\phi(s\cdot)\in \Band$.
\end{enumerate}

%
%
%
\end{lemma}
\begin{proof}
Statement (i) follows readily from  \cite[Lemmas 2.1 and 4.3]{CMT2021}; we remark here for the benefit of the reader that the absolute value of $\phi$ found in \cite{CMT2021} corresponds to $|\phi^\dagger|$ in our current notation (see \cite[Subsection 2.1]{CMT2021} for details).

Fix now $s\in \fT_d$. We infer from (i) that
\[
\phi(sa)=\langle aU^*\xi,U^*\pi(s^*)\eta\rangle, \quad a\in \fA_d
\]
whence (ii) follows.
\end{proof}

We now return to the matter at hand.

\begin{theorem}\label{T:HenAd}
The set $\Band$ is a right band, but it is not a left band.
\end{theorem}
\begin{proof}
Fix $\psi\in \Band$ and $t\in \fT_d$. As observed before, we have that $\Hen_{B(\fF^2_d)}(\fA_d)=\E_{B(\fF^2_d)}(\fA_d)$, so there is a weak-$*$ continuous functional $\omega$ on $B(\fF^2_d)$ such that if we put $\theta=\omega|_{\fT_d}-\psi\in \fT_d^*$, then $\theta\in \fA_d^\perp$. Lemma \ref{L:riesz} implies that $\theta(t\cdot)\in \Band$. On the other hand, it is clear that $\omega(t\cdot )$ is another weak-$*$ continuous functional on $B(\fF^2_d)$, so that 
\[
\psi(t\cdot)=\omega(t\cdot)|_{\fT_d}-\theta(t\cdot)
\]
lies in $\Band$.

 Let now $\phi\in \fT^*_d$ such that $\phi^\dagger\ll\psi^\dagger$. Since $\Hen_{B(\fF^2_d)}(\fA_d)$ is norm closed, so is $\Band$, whence Theorem \ref{T:RN} and the previous paragraph imply that $\phi \in \Band$. We conclude that $\Band$ is indeed a right band.
 
Finally, it follows from \cite[Example 2]{CMT2021} that there is a functional $\tau\in \fA_d^\perp$ such that $\tau^\dagger\notin \Band$. But $|\tau^\dagger|\in \Band$ by Lemma \ref{L:riesz}, and clearly $\tau^\dagger\ll|\tau^\dagger|$, so that $\Band$ is not a left band.
\end{proof}

Next, we give a description of $\Band$. For this purpose, we let $\omega_0:\fT_d\to \bC$ denote the state defined as
\[
\omega_0(t)=\langle te_\varnothing,e_\varnothing\rangle, \quad t\in \fT_d.
\]
Let $\Ro$ denote the convex subset of states $\rho$ on $\fT_d$ such that $\rho|_{\fA_d}=\omega_0|_{\fA_d}$. We also let $\widehat{\Ro}$ denote the norm closure in $\fT_d^*$ of the set 
\[
\{\rho(b^* \cdot b):\rho\in \Ro, b\in \fA_d\}.
\]

\begin{theorem}\label{T:crAd}
We have that 
\[
\AC(\widehat\Ro)^\dagger=\Band=\AC(\Delta_*)^\dagger\subset\AC(\Delta)
\]
where the inclusion is strict.
\end{theorem}
\begin{proof}
The facts that $\Band=\AC(\Delta_*)^\dagger\subset\AC(\Delta)$ and that the inclusion is strict follow from Theorems \ref{T:nccr} and \ref{T:Henband}, along with Lemma \ref{L:analAd} and Theorem \ref{T:HenAd}. 

It thus only remains to show that $\AC(\widehat\Ro)^\dagger=\Band$. For this purpose, let $\phi\in \fA_d^\perp$ be non-zero. By Lemma \ref{L:riesz},  there is a  unital $*$-homomorphism $\pi:\fT_d\to B(\H)$ and a non-zero vector $\xi\in \H$ such that
\[
|\phi^\dagger|(t)=\langle \pi(t)\xi,\xi\rangle
\]
for every $t\in \fT_d$.  Furthermore, if we let $\M=\ol{\pi(\fA_d)\xi}$, then there is a unitary operator $U:\fF^2_d\to \M$ such that
\[
U a U^*=\pi(a)|_\M, \quad a\in \fA_d.
\]
We claim that $|\phi^\dagger|\in \AC(\mathscr{R}_0)$. Indeed, let $\eps>0$ and note that $Ue_\varnothing$ is necessarily cyclic for $\pi(\fA_d)|_\M$.
We may then choose $b\in \fA_d$ with the property that 
\[
\|\pi(b)Ue_\varnothing-\xi\|<\eps.
\]
Define $\rho\in \fT_d^*$ as
\[
\rho(t)=\langle \pi(t)Ue_\varnothing, Ue_\varnothing\rangle, \quad t\in \fT_d.
\]
It is readily verified that $\rho|_{\fA_d}=\omega_0|_{\fA_d}$ so that $\rho\in \Ro$, and that
\[
\|\rho(b^* \cdot b)-|\phi^\dagger|\|<(1+\|\xi\|)\eps.
\]
Since $\eps>0$ was chosen to be arbitrary, we infer that $|\phi^\dagger|\in \widehat \Ro$. We conclude that $\Delta_*\subset \widehat{\Ro}$. Thus, $\Band^\dagger=\AC(\Delta_*)\subset \AC(\widehat \Ro)$.

Conversely, let $\rho\in \Ro$ and $b\in \fA_d$. Put $\theta=\rho(\cdot b)$. Since $\fA_d$ is an algebra and multiplication is separately weak-$*$ continuous on $B(\fF^2_d)$, it is readily seen that $\theta\in \Band$. Put now $\sigma=\theta(b^*\cdot)$. It follows from Theorem \ref{T:RN} that $\sigma^\dagger\ll \theta^\dagger$. 
But $\Band$ is a right band by Theorem \ref{T:HenAd}, whence $\rho(b^*\cdot b)=\sigma\in \Band$. Since $\Band$ is norm closed, we infer that $\widehat\Ro\subset \Band$. Now, $\widehat\Ro$ consists of self-adjoint functionals, so that $\widehat\Ro\subset \Band^\dagger$. Since $\Band$ is a right band, we infer that $\AC(\widehat\Ro)\subset \Band^\dagger$, and the proof is complete.
\end{proof}

We recall here that given a state $\rho$ on some $\rC^*$-algebra $\fT$ and some element $t\in \fT$, typically the functional $\rho(t^*\cdot t)$ is not absolutely continuous with respect to $\rho$; see Example \ref{E:ACmodule}. However, this statement does hold true whenever $\fT$ is commutative by Lemma \ref{L:ACmodule}. Thus, in the commutative setting we find $\AC(\widehat\Ro)=\AC(\Ro)$, hence the previous result can be viewed as a partial analogue of Theorem \ref{T:crAB}.

%


\section{Non-commutative peaking and interpolation}\label{S:nullproj}

In this final section, we give another application of our abstract results from Section \ref{S:ncHenkin}. The setting is that of non-commutative peak sets, as introduced by Hay in \cite{hay2007} and further refined in a series of subsequent papers \cite{BHN2008},\cite{blecher2013},\cite{BR2011},\cite{BR2013}. Our main result in this direction will highlight a web of implications -- some already known, some new -- between numerous fine properties of closed projections. 
\begin{definition}
Let $\fT$ be a unital $\rC^*$-algebra and let $\A\subset \fT$ be a unital norm-closed subalgebra.  Let $q\in \fT^{**}$ be a closed projection. 
We say that $q$ is:
\begin{enumerate}[{\rm (a)}]

\item a \emph{null} projection if $|\phi|(q)=0$ for every $\phi\in \A^\perp$;

\item  a \emph{peak} projection if there is a contraction $a\in \A$ such that $aq=q$ and $\|ap\|<1$ for every closed projection $p\in \fT^{**}$ orthogonal to $q$;

\item  an  \emph{interpolation} projection if $q\A=q\fT$;

\item a \emph{vanishing locus} projection if there is $a\in \A$ such that $aq=0$ and $ap\neq 0$ for every non-zero closed projection $p\in \fT^{**}$ orthogonal to $q$.
\end{enumerate}
Let $\fW$ be a von Neumann algebra which contains a copy of $\fT$. We say that $q$ is a \emph{totally null} projection if $|\phi|(q)=0$ for every $\phi\in \fT^*$ such that $\phi|_\A\in \Hen_\fW(\A)$.
\end{definition}

%
%
%
%
%

\begin{theorem}\label{T:peakequiv}
Let $\fW$ be a von Neumann algebra, let $\fT\subset \fW$ be a unital $\rC^*$-subalgebra, and let $\A\subset \fT$ be a unital norm-closed subalgebra. Let $q\in \fT^{**}$ be a closed projection. Consider the following statements.
\begin{enumerate}[{\rm (i)}]

\item  $q$ is a totally null projection

\item $q$ is a null projection

\item   $q$ is an interpolation projection

\item $q\in \A^{\perp\perp}$

\item $q$ is a peak projection

\item $q$ is a vanishing locus projection
\end{enumerate}
Then, we have
\[
{\rm (i)}\Rightarrow {\rm (ii)}\Leftrightarrow  {\rm (iii)}\Rightarrow  {\rm (iv)}\Leftarrow  {\rm (v)}\Rightarrow  {\rm (vi)}.
\]
If the triple $\A\subset \fT\subset \fW$ is analytic,  then we have  ${\rm (ii)}\Rightarrow  {\rm (i)}$.
If $\A$ is separable, then we  have  ${\rm (iv)}\Rightarrow  {\rm (v)}$.
\end{theorem}
\begin{proof}
(i) $\Rightarrow$ (ii):  This is trivial.

(ii) $\Rightarrow$ (iv):  If $\phi\in \A^\perp$, then $|\phi|(q)=0$ because $q$ is assumed to be null. But then Lemma \ref{L:statevanish} implies that $\widehat\phi(q)=0$, so that $q\in \A^{\perp\perp}$.

(ii) $\Rightarrow$ (iii): This is contained in the proof of \cite[Proposition 3.4]{hay2007}, but we reproduce the details here. Consider $q\fT$ as a subspace of $\fT^{**}$. Let $\phi:q\fT\to \bC$ be a bounded linear functional such that $q\A\subset \ker \phi$. Using that $(q\fT)^{**}$ is isometrically and weak-$*$ homeomorphically isomorphic to $(q\fT)^{\perp\perp}=q\fT^{**}$, we obtain a weak-$*$ continuous linear functional $\phi':q\fT^{**}\to \bC$ extending $\phi$. Define a weak-$*$ continuous linear functional $\psi:\fT^{**}\to \bC$ as
\[
\psi(\xi)=\phi'(q\xi ), \quad \xi\in \fT^{**}.
\]
Since $q\A\subset \ker \phi$, it follows that $\psi|_\fT\in \A^\perp$, whence $|\psi|(q)=0$ since $q$ is assumed to be null. Invoking Lemma \ref{L:statevanish}, we infer that $\psi(q\xi)=0$ for every $\xi\in \fT^{**}$, whence $\phi=0$. By the Hahn--Banach theorem, this implies that $q\A$ is norm dense in $q\fT$. On the other hand, because we showed in the previous paragraph that $q\in \A^{\perp\perp}$, it follows from \cite[Proposition 3.1]{hay2007} that $q\A$ is closed, so that $q\A=q\fT$.

 (iii) $\Rightarrow$ (ii):  Let $\phi\in \A^\perp$ and choose a partial isometry $v\in \fT^{**}$ such that $|\phi|=\widehat\phi(\cdot v)$. By assumption, we have $q\A=q\fT$. Taking weak-$*$ closures in $\fT^{**}$, we obtain $q\A^{\perp\perp}=q\fT^{**}$. Thus, there is $\xi\in \A^{\perp\perp}$ such that $q\xi=qv$. Note then that $\widehat\phi(q\xi)=0$ since $q\xi\in \A^{\perp\perp}$, hence
\[
|\phi|(q)=\widehat\phi(qv)=\widehat\phi(q\xi)=0.
\]
We conclude that $q$ is a null projection.

(v) $\Rightarrow$ (iv): This is an immediate consequence of \cite[Lemma 3.6 and Theorem 5.1]{hay2007}.

(v) $\Rightarrow$ (vi): Since $q$ is a peak projection, there is a contraction $a\in \A$ such that $aq=q$ and $\|ap\|<1$ for every closed projection $p\in \fT^{**}$ orthogonal to $q$. Let $b=I-a\in \A$. We see that $bq=0$ and if $p\in \fT^{**}$ is non-zero closed projection orthogonal to $q$, then
\[
\|bp\|=\|p-ap\|\geq 1-\|ap\|>0
\]
so that $bp\neq 0$. 

Assume now that the triple $\A\subset \fT\subset \fW$ is analytic and that (ii) holds. Thus, $|\phi|(q)=0$ for every $\phi\in \A^\perp$. Clearly, this implies that $\widehat\delta(q)=0$ for every $\delta$ in the norm closure $\Delta\subset \fT^*$ of the convex hull of $\{|\phi|:\phi\in \A^\perp,\|\phi\|=1\}$.  It follows from Lemma \ref{L:statevanish} that $|\phi|(q)=0$ for every $\phi\in \AC(\Delta)$.  But Theorem \ref{T:nccr} implies that $\AC(\Delta)$ contains all functionals whose restriction to $\A$ lies in $\Hen_\fW(\A)$, so indeed $q$ is totally null.

Finally, when $\A$ is separable the fact that (iv) implies (v) is proved in \cite[Theorem 3.4]{BN2012} (alternatively, see \cite[Corollary 3.3]{CTh2020fdim}).
\end{proof}

We remark here that while most of the implications in the previous theorem were already contained (at least implicitly) in the literature, the relationship between a closed projection being null and being totally null  is novel, and requires our main result Theorem \ref{T:nccr}.

Classically, for the triple $\AB\subset \rC(\bS_d)\subset L^\infty(\bS_d,\sigma)$, the closed projection $q$ in the previous theorem is the characteristic function of a closed subset $K\subset \bS_d$. In this case, conditions (i) through (vi) in Theorem \ref{T:peakequiv} are all equivalent \cite[Chapter 10]{rudin2008}. In fact, these conditions are further equivalent to $K$ being a so-called \emph{peak interpolation} set; something similar is true in general \cite[Theorems 3.4 and 5.10]{hay2007}. We will not discuss the corresponding notion of ``peak interpolation projection" here, but we mention that a beautiful non-commutative peak interpolation theory has emerged over the years that parallels the classical one in many ways; see  \cite{BHN2008},\cite{blecher2013},\cite{BR2011},\cite{BR2013},\cite{CD2016duality},\cite{DH2020} for details.

In general, it is not true that (iv) implies (v) unconditionally: even in the classical setting of a uniform algebra, some form of topological regularity is required (see \cite[Lemma 12.1 and Theorem 12.7]{gamelin1969}). We close the paper with an example showing that the implications (vi) $\Rightarrow$ (i), (vi) $\Rightarrow$ (iii) and  (vi) $\Rightarrow$ (v) generally fail as well, even in the separable setting (see \cite{CD2018ideals} for a related discussion). 

\begin{example}\label{E:zeroset}
Let $d\geq 2$ be an integer, and let $H^2_d$ be the Drury--Arveson space on $\bB_d$. In this example, we work with the triple $\A(H^2_d)\subset \fT(H^2_d)\subset B(H^2_d)$.

It follows from Lemma \ref{L:fz} that we may identify $\fT(H^2_d)^{**}$ with $ B(H^2_d)\oplus \rC(\bS_d)^{**}$. Under this identification, $f\in \A(H^2_d)$ corresponds to $M_f\oplus f$; this will be used implicitly below.

Let $K\subset \bS_d$ be a closed set which is a zero set for $\A(H^2_d)$, in the sense that there is $a\in A(H^2_d)$ with $a^{-1}(\{0\})=K$.  Let $\chi_K\in \rC(\bS_d)^{**}$ denote the characteristic function of $K$. It is easily verified that $q_K=0\oplus \chi_K$ is a closed projection in $\fT(H^2_d)^{**}$. We claim that it is a vanishing locus projection. First, it is clear $aq_K=0$ by choice of $a$ and $K$. Next, let $p\in \fT(H^2_d)^{**}$ be a non-zero closed projection orthogonal to $q_K$. If we write $p=p'\oplus p''$ with $p'\in B(H^2_d)$ and $p''\in \rC(\bS_d)^{**}$, then $p''$ is a closed projection relative to  $\rC(\bS_d)$. In particular, there is a closed subset $L\subset \bS_d$ such that $p''=\chi_L$. Using that $pq_K=0$, we find $K\cap L=\varnothing$. If $p''\neq 0$, then $L$ is non-empty and the function $a$ does not vanish on $L$, so $ap''\neq 0$ and $ap\neq 0$. On the other hand, if $p''=0$, then necessarily $p'\neq 0$. Recall now that $a\neq 0$, so the operator $M_a\in B(H^2_d)$ is injective and $M_a p'\neq 0$. We infer that $ap\neq 0$ in that case well, so indeed $q$ is a vanishing locus projection.

Assume further that $K$ is of the type constructed in \cite[Theorem 11.1]{DH2020}. Then,  Lemma \ref{L:absvalmeasure} implies that $q_K=0\oplus \chi_K$ is not totally null. Moreover, 
\[
q_K \A(H^2_d)=\{0\}\oplus \{a|_K:a\in \A(H^2_d)\} 
\]
and
\[
q_K \fT(H^2_d)= \{0\} \oplus C(K).
\]
Hence, we conclude from \cite[Theorem 12.2]{DH2020} that $q_K \A(H^2_d)$ is a proper subset of $q_K \fT(H^2)$, so $q_K$ is not an interpolation projection. Finally, let $b\in \A(H^2_d)$ be a contraction with $bq_K=q_K$. This implies that the function $b$ is identically equal to $1$ on $K$.  Another application  of \cite[Theorem 12.2]{DH2020} yields the existence of a point $\zeta\in \bS_d\setminus K$ such that $b(\zeta)=1$. Then, the projection $p=0\oplus \chi_{\zeta}\in \fT(H^2_d)^{**}$ is closed and orthogonal to $q_K$, and
$
\|bp\|=b(\zeta)=1.
$
We infer that $q_K$ is not a peak projection.
\qed
\end{example}


\bibliography{biblio_ncHenkin}
\bibliographystyle{plain}


\end{document}